\documentclass[a4paper, 11pt]{article}
\usepackage{geometry}
\usepackage[utf8]{inputenc}
\usepackage[english]{babel}
\usepackage{amsmath,amssymb,amsthm,graphicx,mathrsfs,url}

\usepackage{authblk}

\usepackage{dsfont} 
\usepackage{amsmath,amssymb,amsthm}
\usepackage[justification=centering]{caption}
\usepackage[colorlinks=true, citecolor=green, linkcolor=blue]{hyperref}
\usepackage[nameinlink, capitalize]{cleveref}

\usepackage[colorinlistoftodos]{todonotes}
\usepackage{varwidth}

\usepackage{color}
\usepackage{float}
\usepackage[super]{nth}
\usepackage{booktabs}
\usepackage{rotating}
\usepackage{verbatim}
\usepackage{multirow}
\usepackage{physics}
\usepackage{mathtools}
\usepackage{subcaption}
\usepackage{bm}
\usepackage[toc,page]{appendix}
\usepackage[shortlabels]{enumitem}
\usepackage{nomencl}

\numberwithin{equation}{section}

\newcounter{mnotecount}[section]

\theoremstyle{plain}
\newtheorem{theorem}{Theorem}

\newtheorem{lemma}{Lemma}
\newtheorem{proposition}{Proposition}

\theoremstyle{definition}

\newtheorem{remark}{Remark}

\title{Local Wellposedness and Global Weak Solutions of the Pauli-Darwin/Poisswell Equations}

\author[a]{Pierre Germain}
\author[b]{Norbert J. Mauser}
\author[b,c]{Jakob Möller}
\affil[a]{Imperial College London, Dept of Mathematics, South Kensington Campus, London SW7 2AZ, UK}
\affil[b]{Research Platform MMM "Mathematics-Magnetism-Materials" c/o Fak. Mathematik, Univ. Wien, Oskar Morgenstern Platz 1, A-1090 Vienna}
\affil[c]{CMLS, École Polytechnique, F-91128 Palaiseau}

\begin{document}
\maketitle
%
\begin{abstract}
We construct local (in time) strong solutions in {$H^s(\mathbb{R}^3)$, $s>3/2$} and global weak solutions with finite energy for both the Pauli-Darwin and the Pauli-Poisswell systems. These are the first rigorous results on local and global wellposedness for these nonlinear first-order semi-relativistic quantum models for fast moving electrons. The Pauli equation is essentially a vector-valued magnetic Schrödinger equation for a 2-spinor with an additional Stern-Gerlach term coupling spin and magnetic field, keeping terms up to first order in $1/c$, where $c$ denotes the speed of light. The self-consistent electromagnetic field is computed from the charge density and current density by semi-relativistic approximations of the Maxwell equations: the Poisswell equation at $O(1/c)$ and the Darwin equation at $O(1/c^2)$.\\ We present the physics and asymptotic relations and provide proofs that rely on energy estimates for the strong solutions and compactness with an appropriate regularization for the weak solutions.
\end{abstract}
%

\setcounter{tocdepth}{1}
\tableofcontents

\section{Introduction}
\subsection{The equations}

 We consider two nonlinear systems of PDE consisting of the Pauli equation, introduced by Wolfgang Pauli in \cite{pauli1927quantenmechanik}, for the unknown 2-spinor $u$ with values in $\mathbb{C}^2$, coupled to semi-relativistic approximations of the Maxwell equations for the unknown scalar electric potential $V$ (taking values in $ \mathbb{R}$) and the vector-valued magnetic potential $A$ (taking values in $ \mathbb{R}^3$), where the source terms are the scalar charge density $|u|^2 = |u_1|^2+|u_2|^2$ and the vector-valued current density $J$.
 
 We consider the setting in 3 space dimensions, i.e. $x \in \mathbb{R}^3$, which is the physically meaningful case when magnetic fields are involved. 
 
The scaled dimensionless Pauli-Darwin equation in Coulomb gauge \eqref{eq:coulomb_} is given by
\begin{align}
    i\partial_t u &= -\frac{1}{2} \Delta_A u + V u -\frac{1}{2}  (\mathbf{\sigma} \cdot B) u, \label{eq:PD_Pauli}\\
    -\Delta V &= |u|^2, \label{eq:PD_PoissonV} \\
    -\Delta A&= J-\partial_t \nabla V, \label{eq:PD_PoissonA}
    \qquad B = \nabla \times A \\
    \nabla \cdot A &= 0, \label{eq:coulomb_} \\
        u(0,x) &= u_0(x).  \label{eq:PD_data}
\end{align}
The scaled dimensionless Pauli-Poisswell equation in Lorenz gauge \eqref{eq:lorenz} is given by
\begin{align}
    i\partial_t u &= -\frac{1}{2} \Delta_A u + V u -\frac{1}{2}  (\mathbf{\sigma} \cdot B) u, \label{eq:PPw_Pauli}\\
    -\Delta V &= |u|^2, \label{eq:PPW_PoissonV} \\
    -\Delta A&= J\label{eq:PPW_PoissonA}, \qquad B = \nabla \times A, \\
    \nabla \cdot A + \partial_t V &= 0, \label{eq:lorenz} \\
        u(0,x) &= u_0(x)  \label{eq:PPW_data},
\end{align}
For both equations, the Pauli current density is given by
\begin{equation}
    J(u,A) = \Im \langle u, \nabla_A u\rangle + \frac{1}{2}\nabla \times \langle u, \mathbf{\sigma} u \rangle\label{eq:PD_current}
\end{equation}
and we used the following notations
\begin{itemize}
\item $\Delta_A$ stands for the magnetic Laplace operator $\nabla_A^2 = (\nabla-i A)^2 = \sum_{k=1}^3 (\partial_k-i A_k)^2$.
\item $ \mathbf{\sigma} \cdot B = \sum_{k=1}^3 \sigma_k B_k$, where the $\{\sigma_i\}$ are the Pauli matrices (cf. the notations in Section \ref{sec:notations}).
\item $\langle \cdot , \cdot \rangle = \langle \cdot , \cdot \rangle_{\mathbb{C}^2}$ denotes the inner product in $\mathbb{C}^2$ (antilinear in the first variable).
\item $\langle u, \nabla u \rangle$ is the 3-vector with components $\langle u,\partial_k u\rangle$, $k=1,2,3$. 
\item Similarly, $\langle u, A u \rangle$ and $\langle u, \sigma u \rangle$ are the 3-vectors with entries $\langle u,A_k u\rangle$ and $\langle u, \sigma_k u\rangle$.
\end{itemize}

As we shall see below, the above system implies the continuity equation
\begin{equation}
    \partial_t |u|^2 +\nabla \cdot J = 0,
\end{equation}
Note that using this relation, the term
$J-\partial_t \nabla V$ in the Pauli-Darwin equation can be recast as $\mathbb{P}J$, where $\mathbb{P}$ denotes the Leray projection, i.e. the Fourier multiplier with matrix-valued symbol
\begin{equation}
    m_{ij}(\xi) = \delta_{ij} - \frac{\xi_i \xi_j }{|\xi|^2}, \quad 1\leq i
    ,j\leq 3,
\end{equation}
which is a bounded operator $\mathbb{P}\colon L^p \rightarrow L^p$ for all $1<p<\infty$. The Leray projection maps onto divergence free vector fields, i.e. $\nabla \cdot (\mathbb{P}J) = 0$. \\

The reader will find in Section \ref{sectionphysics} a discussion of the physical relevance and the regime of validity of these equations, which arise as singular limits of the Dirac-Maxwell equation.

\subsection{Main results}

Our main results are the following two theorems on the local wellposedness
and the existence of global weak solutions 
that are, as far as we know, the first rigorous results on the existence of solutions of the Pauli-Darwin and Pauli-Poisswell equations.

\begin{theorem}[\textbf{Local wellposedness}]
Let {$s>3/2$} and $u_{0} \in H^s(\mathbb{R}_x^3)$. Then there exists $T>0$ such that the Pauli-Darwin equation \eqref{eq:PD_Pauli}-\eqref{eq:PD_data} and the Pauli-Poisswell equation \eqref{eq:PPw_Pauli}-\eqref{eq:PPW_data} have unique solutions  $u^{\text{\emph{D}}},u^{\text{\emph{PW}}} \in C([0,T],H^s(\mathbb{R}_x^3))$ with initial data $u_{0}$. The solutions depend continuously on the data in the $H^s$ topology.
\label{thm:local wellposedness}
\end{theorem}

\begin{remark} In order to gauge the strength of this local well-posedness result, it is instructive to examine how the equations behave under a rescaling of $x$ and $t$. Both equations turn out not to be scale invariant, but the most singular interactions can be captured in the following caricature
$$
i \partial_t u - \Delta u = \nabla (\Delta^{-1} (u \nabla u) u)
$$
which admits the scaling invariance $u(t,x) \longrightarrow \lambda u (\lambda^2 t , \lambda x)$. This means that $H^{1/2}$ is scale invariant, and that the equations should naturally be expected to be locally well-posed on $H^s$ with $s > \frac 12$. This regularity might indeed be reached with the help of more advanced tools; this will be the object of future work.
\end{remark}

\begin{theorem}[\textbf{Global weak solutions}]
\label{th:global weak solutions}
    Let $u_0 \in H^1(\mathbb{R}_x^3)$. Then the Pauli-Darwin equation \eqref{eq:PD_Pauli}-\eqref{eq:PD_data} {and the Pauli-Poisswell equation \eqref{eq:PPw_Pauli}-\eqref{eq:PPW_data}} each have at least one weak solution such that $u^{\text{D}}{, u^{\text{PW}}}  \in BC_{w}(\mathbb{R}_t,H^1(\mathbb{R}_x^3))$, where $BC_{w}$ denotes the space of bounded weakly continuous functions.
\end{theorem}

\subsection{The Schrödinger-Maxwell equation}

While there are no results on the wellposedness of the Pauli-Darwin and Pauli-Poisswell equations, the magnetic Schrödinger-Maxwell equation has been studied in the mathematical literature.
For scalar $\psi\in L^2(\mathbb{R}^d,\mathbb{C})$ it is given by
\begin{align}
& i \partial_t \psi = -\frac{1}{2}\Delta_A \psi + V \psi,
\label{eq:Schrödinger_Maxwell} \\
& -\Delta  V +\frac{1}{c}\partial_t(\nabla \cdot A )= |\psi|^2 \label{eq:Maxwell1}\\
& (\frac{1}{c}\partial^2_t-\Delta)  A + \nabla(\nabla \cdot A + \frac{1}{c}\partial_t V) =  \frac{1}{c}J \label{eq:Maxwell2},
\end{align}
where 
\begin{equation}
  J = \Im(\overline{\psi}(\nabla-i A)\psi)
  \label{eq:SMcurrent}
\end{equation} 
is the current density of the magnetic Schrödinger equation. Note that no gauge is fixed in \eqref{eq:Maxwell1}-\eqref{eq:Maxwell2}.

The Schrödinger-Maxwell equation is inconsistent in $1/c$: In \eqref{eq:Schrödinger_Maxwell}-\eqref{eq:SMcurrent}, the Lorentz invariant, fully relativistic Maxwell equations \eqref{eq:Maxwell1}-\eqref{eq:Maxwell2} are coupled to the Galilei invariant magnetic Schrödinger equation \eqref{eq:Schrödinger_Maxwell}, which is $O(1)$ non-relativistic since it lacks the $O(1/c)$ spin term. The first result on global wellposedness was established in 1986 by Nakamitsu and Tsutsumi in \cite{nakamitsu1986cauchy} using energy estimates. In 1995, Guo, Nakamitsu and Strauss obtained  global weak solutions in the energy space without uniqueness by parabolic regularization, cf. \cite{guo1995global}. Global wellposedness above the energy space was proved by Nakamura and Wada in 2005 and 2007 in \cite{nakamura2005local, nakamura2007global} using refined Koch-Tzvetkov estimates. This was improved by Bejenaru and Tataru in 2009 to the energy space $H^1$ in \cite{bejenaru2009global} by constructing a parametrix. In 2003, Ginibre and Velo studied the  scattering of the Schrödinger-Maxwell system, cf. \cite{ginibre2003long}. An extended model with additional power nonlinearity was studied by Antonelli, Marcati and Scandone 2019 in \cite{antonelli2019global}.

The Schrödinger-Maxwell equation with spin but without the Poisson part for the electric potential was studied by Kieffer in 2020 in \cite{kieffer2020time}.

\subsection{Organization of the article}

In section \ref{physics} we present the modeling of the equations as first order semi-relativistic approximations of the fully relativistic Dirac-Maxwell system. The notations and some useful formulas used in the paper are given in section \ref{notations}. In section \ref{sec:preliminaries} we present the energy structure and the regularized equation, for which the energy decays. We collect some elliptic estimates for the Poisson equations in section \ref{Elliptic estimates}. Local wellposedness of the regularized equation is then established in section \ref{Parabolic regularization}. Theorem \ref{thm:local wellposedness} is proved in section \ref{sec:lwp} via an a priori estimate in $H^s$, $s>3/2$. Theorem \ref{th:global weak solutions} is proved in section \ref{sec:weak sol} by compactness.

\section{The Physics}
\label{physics}

\subsection{The equations including the physical dimensions}

\label{sectionphysics}

The Pauli-Darwin and Pauli-Poisswell equations were written above in a more concise form by setting all constants equal to 1. But in order to discuss the physical relevance of these equations, physical constants are of course a key element.

For a fermion of mass $m$ and charge $q$, the \emph{Pauli-Darwin equation} in Coulomb gauge \eqref{eq:PD_Coulomb_scaled}, cf. \cite{moller2023models}, for the 2-spinor $u=(u_1,u_2)^\top$ depending on $(t,x) \in \mathbb{R} \times \mathbb{R}^3$ is given by
\begin{align}
    i\hbar\partial_t u &= -\frac{1}{2m}(\hbar\nabla-i\frac{q}{c}A)^2u + qV u - \frac{q\hbar}{2mc} (\sigma \cdot B) u, \label{eq:PD_Pauli_scaled}\\
    -\Delta V &= |u|^2, \label{eq:PD_PoissonV_scaled}\\
    -\Delta A &= \frac{1}{c}(J-\partial_t \nabla V), \qquad B = \nabla \times A, \label{eq:PD_PoissonA_scaled} \\
    \nabla\cdot A &= 0, \label{eq:PD_Coulomb_scaled}
    \\   u(0,x) &= u_0^{\text{D}}(x). \label{eq:PD_data_scaled}
\end{align}
The \emph{Pauli-Poisswell equation} in Lorenz gauge \eqref{eq:PPW_Lorenz_scaled}, cf. \cite{masmoudi2001selfconsistent}, is given by
\begin{align}
    i\hbar\partial_t u &= -\frac{1}{2m}(\hbar\nabla-i\frac{q}{c}A)^2u + qV u - \frac{q\hbar}{2mc} (\sigma \cdot B) u, \label{eq:PPW_Pauli_scaled}\\
    -\Delta V &= |u|^2, \label{eq:PPW_PoissonV_scaled}\\
    -\Delta A &= \frac{1}{c}J, \qquad B = \nabla \times A, \label{eq:PPW_PoissonA_scaled} \\
    \nabla\cdot A +\frac{1}{c}\partial_t V &= 0, \label{eq:PPW_Lorenz_scaled}
    \\   u(0,x) &= u_0^{\text{PW}}(x). \label{eq:PPW_data_scaled}
\end{align}
For both equations, the \emph{Pauli current density} is given by
\begin{equation}
    J(u,A) =\frac{1}{m}\Im \langle u , (\hbar\nabla -i\frac{q}{c}A)u \rangle +\frac{\hbar}{2m}\nabla \times \langle u , \sigma u \rangle, \label{eq:PD_current_scaled}
\end{equation}
The Pauli current density \cite{nowakowski1999quantum} extends the current density $\hbar\Im\overline{\psi}\nabla \psi$ of the Schrödinger equation for a scalar wave function $\psi$ not only to the magnetic Schrödinger current density $\Im\overline{\psi}(\hbar\nabla-i(q/c)A) \psi$, but also includes a divergence-free spin term, which couples the two spin components.

The quantity $|u|^2= \langle u,u\rangle= |u_1|^2+|u_2|^2$ is the scalar charge density. $V\colon \mathbb{R}_t\times \mathbb{R}^3_x\rightarrow \mathbb{R}$ and $A\colon \mathbb{R}_t\times \mathbb{R}^3_x\rightarrow \mathbb{R}^3$  are the self-consistent electric and magnetic potentials, coupled to the Pauli equation via Darwin's equations \eqref{eq:PD_PoissonV_scaled}-\eqref{eq:PD_PoissonA_scaled}, cf. \cite{degond1993analysis, krause2007unified, shiroto2023improved}, and the Poisswell equations \eqref{eq:PPW_PoissonV_scaled}-\eqref{eq:PPW_PoissonA_scaled}, cf. \cite{besse2007numerical, masmoudi2001selfconsistent}, with the charge density and current density as sources. Note that in Darwin's equations the source for $A$ is the {transversal} Pauli current density $J-\partial_t \nabla V$, which appears due to the Coulomb gauge.

\begin{remark}
    Note that in \cite{masmoudi2001selfconsistent}, there is a factor $1/2$ missing in front of the curl term of the current \eqref{eq:PD_current_scaled} and in \cite{moller2024pauli} the sign in front of the same term is wrong.
\end{remark}

\subsection{Semi-relativistic approximation}

The self-consistent {Pauli-Darwin} and Pauli-Poisswell equations for a 2-spinor were introduced in \cite{masmoudi2001selfconsistent, moller2023models}  and are $O(1/c)$ semi-relativistic approximations of the relativistic Dirac-Maxwell equation for a 4-spinor in {Coulomb gauge} and {Lorenz gauge}, respectively. They are first-order semi-relativistic, self-consistent models for an electrically charged spin-$1/2$-particle with self-interaction with the electromagnetic field and spin, where the 2-spinor $u=(u_1,u_2)^{\top}$ represents the two spin states $u_1,u_2$ of a fermion and corresponds to the electron component of the Dirac spinor, obtained by a Foldy-Wouthuysen transform and projection on the upper component, cf. \cite{bechouche1998semi, bechouche2005asymptotic, foldy1950dirac, masmoudi2003nonrelativistic, mauser1999rigorous, mauser2000semi}. It is important to note here that both models are $O(1/c)$. Even though the Darwin approximation is $O(1/c^2)$, the Pauli-Darwin system is $O(1/c)$ since the Pauli equation is $O(1/c)$.

In fact, it is possible to derive a fully $O(1/c^2)$ model, where Darwin's equations are coupled to the Pauli equation containing all $O(1/c^2)$ terms, cf. \cite{itzykson2012quantum, mauser1999rigorous, moller2023thesis}. This model, which we call \emph{Pauli$_2$-Darwin equation}, is given by
\begin{align}
\begin{split}
i\hbar\partial_t u =& -\frac{1}{2m}(\hbar \nabla-i\frac{q}{c}A)^2u + {q}V u - \frac{q\hbar}{2mc} (\sigma \cdot B) u, \\ 
    &-\frac{\hbar^4}{8m^3c^2}\Delta^2 u -\frac{q\hbar }{4m^2c^2}\left(\frac{i}{2}\sigma \cdot (\nabla \times E) - (\sigma \cdot (E \times  \nabla))\right)u \\
    &- \frac{\hbar q}{8m^2c^2}(\nabla \cdot E)u, \qquad E =-\nabla V-\frac{1}{c}\partial_t A, \quad B=\nabla \times A \label{eq:Pauli2}
    \end{split}\\
    -\Delta V =& |u|^2, \qquad
    -\Delta A = \frac{1}{c}\mathbb{P}J. 
\end{align}

Returning to Pauli-Poisswell/Darwin, the self-consistent electromagnetic potentials $V$ and $A$ are given by a magnetostatic, i.e. slow with respect to $1/c$, approximation of Maxwell's equations via $1+3$ Poisson equations with density and current density as source terms, \cite{besse2007numerical, krause2007unified, shiroto2023improved}.

The models include:
\begin{enumerate}[(\alph*),topsep=0pt,itemsep=-1ex,partopsep=1ex,parsep=1ex]
    \item the \emph{magnetic field} in the "magnetostatic" $O(1/c^2)$ (for Darwin) and $O(1/c)$ (for Poisswell) approximations by 3 Poisson type equations with the (transversal) current density,
\item the \emph{spin} of the particle, coupled to the magnetic field $B$ by the Stern-Gerlach term in the  Pauli equation as the $O(1/c)$ approximation of the Dirac equation.
\end{enumerate} 

Note that the Pauli equation does not include the antiparticle, which appears intrinsically in the Dirac spinor. The negative part of the energy of the Dirac equation which corresponds to the positron is a major obstacle for obtaining global wellposedness (local wellposedness can be obtained by exploiting a null structure, cf. \cite{d2010null, masmoudi2003uniqueness}, illposedness below charge (i.e. $L^2$) regularity was shown in \cite{selber2021ill}).\\

The key semi-relativistic feature of the Pauli equation is the \emph{spin-magnetic Laplacian} or \emph{Pauli operator}, acting on $L^2(\mathbb{R}\times\mathbb{R}^3,\mathbb{C}^2)$,
\begin{equation}
   \frac{1}{2m} (\sigma \cdot \nabla_A)^2 = \frac{1}{2m}(\sigma\cdot(\hbar\nabla-i\frac{q}{c}A))^2 = \frac{1}{2m}(\hbar\nabla-i\frac{q}{c}A)^2 + \frac{\hbar q}{2mc}(\sigma\cdot B)u,
    \label{eq:spin_magnetic_laplacian}
\end{equation}
which can be written as the sum of the spinless \emph{magnetic Laplacian}, occurring in the usual magnetic Schrödinger equation,
\begin{equation}
    \frac{1}{2m}\Delta_A = \frac{1}{2m}(\hbar\nabla-i\frac{q}{c}A)^2,
    \label{eq:magnetic_laplacian}
\end{equation}
and the \emph{Stern-Gerlach term}  
\begin{equation}
    \frac{\hbar q}{2mc}(\sigma\cdot B)u,
    \label{eq:stern_gerlach_term}
\end{equation}
coupling spin and magnetic field, see also Section \ref{sec:useful_identities}. The other semi-relativistic feature is the spin term of the current density
\begin{equation}
    \frac{\hbar}{2m}\nabla \times \langle u ,\sigma u\rangle,
\end{equation}
which was derived in the $O(1/c)$ approximation of the Dirac current density in \cite{nowakowski1999quantum}. Note that its addition doesn't change the continuity equation since $\text{div curl} =0$. Also note that in \eqref{eq:spin_magnetic_laplacian}, $|A|^2$ is $O(1/c^2)$.

\begin{remark}{\underline{Caveat on "Darwin":}} Equations \eqref{eq:PD_PoissonV}-\eqref{eq:PD_PoissonA} and \eqref{eq:PD_PoissonV_scaled}-\eqref{eq:PD_PoissonA_scaled} are referred to as \emph{Darwin's equations} and are the $O(1/c^2)$ semi-relativistic approximation of Maxwell's equations. The \emph{Poisswell equations} \eqref{eq:PPW_PoissonV}-\eqref{eq:PPW_PoissonA} and \eqref{eq:PPW_PoissonV_scaled}-\eqref{eq:PPW_PoissonA_scaled}  are the $O(1/c)$ semi-relativistic approximation of Maxwell's equations, cf. \cite{besse2007numerical, krause2007unified, moller2023darwin, pallard2006initial, seehafer2009local, shiroto2023improved}. 

On the other hand, the \emph{Darwin term} in relativistic quantum mechanics refers to a second order term in the semi-relativistic approximation of the Dirac equation, i.e. the Pauli$_2$ equation \eqref{eq:Pauli2} and is related to the \emph{Zitterbewegung}, cf, \cite{itzykson2012quantum, mauser1999rigorous, mauser2000semi}. The Pauli equation as the the first order semi-relativistic approximation of the Dirac equation does \emph{not} include the Darwin term.
\end{remark}


\subsection{From Maxwell to Darwin and Poisswell} Maxwell's equations for the potentials with no gauge fixed are given by
\begin{align}
    & -\Delta  V +\frac{1}{c}\partial_t(\nabla \cdot A )= |u|^2 \label{eq:MaxwellV}\\
& (\frac{1}{c^2}\partial^2_t-\Delta)  A + \nabla(\nabla \cdot A + \frac{1}{c}\partial_t V) =  \frac{1}{c}J \label{eq:MaxwellA}
\end{align}
\textbf{Darwin's equations}, i.e. the second order approximation of Maxwell's equations in Coulomb gauge, are given by
\begin{align}
    - \Delta V &= |u|^2 \label{Darwin}\\
    - \Delta A & = \frac{1}{c}(J-\partial_t \nabla V) \\
    \nabla \cdot A &= 0 \label{eq:coulomb}
\end{align}
where the term $\partial_t \nabla V$ corresponds to the \emph{longitudinal current} $J_L$ and $J_T= J-\partial_t \nabla V$ corresponds to the \emph{transversal current} $J_T$. Note that $\nabla \times J_L = 0$ and $\nabla \cdot J_T = 0$.  Formally, Darwin's equations can be obtained
by writing Maxwell's equations \eqref{eq:MaxwellV}-\eqref{eq:MaxwellA} in Coulomb gauge,
\begin{align}
    -\Delta V = |u|^2, && -\Delta A + \frac{1}{c^2}\partial_t^2 A = \frac{1}{c}(J - \partial_t \nabla V) && \nabla \cdot A = 0, 
\end{align}
and, in order to obtain a $O(1/c^2)$ approximation, dropping $\tfrac{1}{c^2}\partial_t^2 A$, which is $O(1/c^3)$, due to $A$ being $O(1/c)$. 

By defining the longitudinal and transversal electric fields $E_L =-\nabla V$ and $E_T = \tfrac{1}{c}\partial_t A$ we obtain the following form of Darwin's equations:
\begin{align}
\label{eq:Darwin fields 1}
    \nabla \cdot E_L = |u|^2, && \nabla \cdot B = 0, \\
    \nabla \times B - \frac{1}{c}\partial_t E_L = J, && \nabla \times E_T + \frac{1}{c}\partial_t B = 0. \label{eq:Darwin fields 2}
\end{align}
Consider the formal asymptotic expansion 
\begin{align}
E = E^0 + \frac{1}{c} E^1 + O(\frac{1}{c^2}), && B = B^0 + \frac{1}{c} B^1 +O(\frac{1}{c^2}), \\
V = V^0 + \frac{1}{c} V^1 + O(\frac{1}{c^2}), && A = A^0 + \frac{1}{c} A^1 +O(\frac{1}{c^2}).
\end{align} 
Then Maxwell's equations at $O(1)$ are given by:
\begin{align}
    \nabla \times E^0 &= 0, & \nabla \cdot E^0 &= |u|^2 \\
    \nabla \times B^0 &= 0, & \nabla \cdot B^0 &= 0 \\
    \nabla \cdot A^0 &= 0 & E^0 &= -\nabla V^0
\end{align}
Therefore
\begin{align}
   -\Delta V^0 = |u|^2 && B^0 \equiv A^0 \equiv 0 
\end{align}
At $O(1/c)$:
\begin{align}
    \nabla \times E^1 &= 0, & \nabla \cdot E^1 &= 0 \\
    -\partial_t E^0 + \nabla \times B^1 &= J, & \nabla \cdot B^0 &= 0 \label{1}\\
    \nabla \cdot A^1 &= 0 & B^1 &= \nabla \times A^1\label{2}
\end{align}
Using $\nabla \times (\nabla \times A^1) = \nabla (\nabla \cdot A^1) - \Delta A^1$, \eqref{1} and \eqref{2}, we obtain for $A^1$:
\begin{align}
    -\Delta A^1 = J- \partial_t \nabla V^0
\end{align}
Therefore $V = V^0 + O(1/c^2)$, $A = A^1 + O(1/c^2)$, which yields the Darwin equations \eqref{Darwin}-\eqref{eq:coulomb}.

The \textbf{Poisswell equations} are given by
\begin{align}
    -\Delta V  &= |u|^2, \\
    -\Delta A &= \frac{1}{c}J,  \\
    \nabla \cdot A + \frac{1}{c}\partial_t V &= 0,
\end{align}
which can be formally obtained by writing Maxwell's equations \eqref{eq:MaxwellV}-\eqref{eq:MaxwellA} in Lorenz gauge,
\begin{align}
    -\Delta V + \frac{1}{c^2} \partial_t^2 V = |u|^2 && -\Delta A +\frac{1}{c^2}\partial_t^2 A = J  && \frac{1}{c}\partial_t V + \nabla \cdot A = 0.
\end{align}
and dropping $\tfrac{1}{c^2} \partial_t^2 V$ and $\tfrac{1}{c^2}\partial_t^2 A$. 
The Poisswell equations can be similarly derived by a formal asymptotic expansion as above together with the Lorenz gauge instead of the Coulomb gauge. Evidently, this does not change Maxwell's equations for the fields \eqref{eq:Darwin fields 1}-\eqref{eq:Darwin fields 2}, cf. \cite{shiroto2023improved}. There is a $O(1/c^2)$ term hidden in the Poisswell approximation: The transversal electric field $E_L = \tfrac{1}{c}\partial_t A$ is $O(1/c^2)$ since $A$ is $O(1/c)$. In the Poisswell approximation, it satisfies the following Poisson equation, cf. \cite{besse2007numerical}:
\begin{equation}
    -\Delta (\partial_t A) = \frac{1}{c} \partial_t J
\end{equation}

In the Darwin approximation no $O(1/c^2)$ terms are missing since the equation for the electric potential in the Coulomb gauge is $-\Delta V = |u|^2$ (hence the $O(1/c^2)$ term $\tfrac{1}{c^2}\partial_t^2 V$ does not appear) and in the equation for the magnetic potential the term $\tfrac{1}{c^2}\partial_t^2 A$  is $O(1/c^3)$ and can be dropped. In the Poisswell approximation, the equation for the electric potential in Lorenz gauge is $-\Delta V +\tfrac{1}{c^2}\partial_t^2 V = |u|^2$ and one drops the $O(1/c^2)$ term $\tfrac{1}{c^2}\partial_t^2 V$. Therefore the Poisswell approximation is \emph{not} $O(1/c^2)$. However, there are $O(1/c^2)$ terms present, as explained above.

\begin{remark}
    In the semiclassical limit the Pauli-Darwin equation converges to the Vlasov-Darwin equation for the Wigner measure or the Euler-Darwin equation for monokinetic Wigner measures. Respectively, the semiclassical limit of the Pauli-Poisswell equation is the Vlasov-Poisswell equation and the Euler-Poisswell equation, cf. \cite{moller2023darwin, yang2023semi}. The Vlasov-Darwin equation was studied in \cite{bauer2005, hankwan2018, pallard2006initial, seehafer2008} and the Vlasov-Poisswell equation in \cite{besse2007numerical, seehafer2009local}. The semiclassical limit of the related Pauli-Poisson model to the magnetic Vlasov-Poisson equation, where the magnetic field is external, was shown in \cite{moller2024pauli}.
\end{remark}

\section{Notations and formulas}
\label{notations}
\label{sec:notations}

\subsection{Inequalities} For two quantities $A,B$, we write $A \lesssim B$ if there exists a constant $C$ such that $A \leq CB$. 
We write $A \sim B$ if $A \lesssim B$ and $B \lesssim A$.  

We will also write $A \leq F(B)$ to mean that there exists a continuous function $F$ such that the quantities $A$ and $B$ satisfy this inequality. Note that the function $F$ can differ from one occurrence to the next.

\subsection{Fourier multipliers and Sobolev spaces}

For $s$ a real number, the operator $D^s$ is the Fourier multiplier with symbol $|\xi|^s$.

The homogeneous Sobolev space $\dot{W}^{s,p}$  of order $s$ is defined by the norm $\|D^s f\|_{L^p}$ and  the inhomogeneous Sobolev space $W^{s,p}$  of order $s$ defined by the norm $ \|\langle D \rangle^s f \|_{L^p}$ where $\langle \xi \rangle = (1+|\xi|^2)^{1/2}$. If $p=2$, we denote $\dot H^s$ and $H^s$ instead of $\dot W^{s,p}$ and $W^{s,p}$.

\subsection{Linear and multilinear operations}
If $v,w \in \mathbb{C}^2$, their inner product is defined as
\begin{equation*}
    \langle v , w\rangle_{\mathbb{C}^2} = \langle v , w\rangle = \overline{v_1} w_1 + \overline{v_2} w_2 .
\end{equation*}
By contrast, if $a,b \in \mathbb{C}^3$, we denote $a \cdot b$ for
$$
a \cdot b = \sum_{i=1}^3 a_i b_i.
$$
The $L^2$ inner product of $v$ and $w$ valued in $\mathbb{C}^2$ is denoted by
$$
(v,w)_{L^2} = (v,w) = \int \langle v,w \rangle \dd x.
$$
We will also denote
$$
a \cdot \nabla = \sum_{i=1}^3 a_i \partial_i, \qquad a \cdot \sigma = \sum_{i=1}^3 a_i \sigma_i. 
$$
Here, the Pauli matrices are given by
\begin{align*}
    \sigma_1 = \begin{pmatrix}
    0 & 1 \\ 1 & 0
    \end{pmatrix}, && 
     \sigma_2 = \begin{pmatrix}
    0 & -i \\ i & 0
    \end{pmatrix}, &&
     \sigma_3 = \begin{pmatrix}
    1 & 0 \\ 0 & -1
    \end{pmatrix}
\end{align*}
They obey the following product laws
\begin{equation}
\label{productlaws}
\begin{split}
& \sigma_i^2 = \operatorname{Id} \qquad \mbox{for $i=1,2,3$}\\
& \sigma_i \sigma_j =  i \sum_{k=1}^3 \varepsilon_{ijk} \sigma_k \qquad \mbox{for $i \neq j$},
\end{split}
\end{equation}
where $\varepsilon_{ijk}$ denotes the Levi-Civita symbol.

\subsection{Useful identities}
\label{sec:useful_identities}

We define the magnetic gradient \begin{equation}
\nabla_A = \nabla-iA = (\partial_k -i A_k)_{k=1,2,3},
\end{equation}
the magnetic Laplacian
\begin{equation}
     \Delta_A = (\nabla_A)^2 = \sum_{k=1}^3 (\partial_k - iA_k)^2 = \Delta-2iA\cdot \nabla  - i(\text{div} A) -|A|^2
\end{equation}
and the spin-magnetic Laplacian
\begin{equation}
 (\sigma \cdot \nabla_A)^2 = \left( \sum_{k=1}^3 \sigma_k (\partial_k - iA_k) \right)^2. 
\end{equation}
It can be written as
$$
 (\sigma \cdot \nabla_A)^2 = \Delta_A + \sigma \cdot B.
$$
To prove this identity, we proceed as follows:
\begin{align*}
(\sigma \cdot \nabla_A)^2 & = \sum_{k,\ell} \sigma_k \sigma_\ell (\partial_k - iA_k)(\partial_\ell - iA_\ell) \\
& = \sum_k (\partial_k - iA_k)^2 + i \sum_{k,\ell,m} \epsilon_{k\ell m} \sigma_m (\partial_k - iA_k)(\partial_\ell - iA_\ell) \\
& = \Delta_A + \sum_{k,\ell,m} \epsilon_{k\ell m} \sigma_m \partial_k A_\ell = \Delta_A + \sigma \cdot B
\end{align*}
Here, we used \eqref{productlaws} in the second equality and the cancellations following from the antisymmetry of the Levi-Civita symbol in the third equality.

Note that the spin-magnetic Laplacian (just like the magnetic Laplacian) is symmetric with respect to the $L^2$ scalar product:
\begin{equation}
    \int_{\mathbb{R}^3} \langle f,(\sigma \cdot \nabla_A)^2 g\rangle \dd x = \int_{\mathbb{R}^3} \langle(\sigma \cdot \nabla_A)^2 f, g\rangle \dd x.
\end{equation}


Finally, the current density can be written as
\begin{align*}
    J(u,A) = \Im \langle \sigma u, (\sigma \cdot \nabla_A) u \rangle = \Im \langle \sigma_j u, \sum_{k=1}^3 \sigma_k(\partial_k -iA_k) u \rangle. \label{eq:J_k}
\end{align*}
Indeed, using once again the properties of Pauli matrices,
\begin{align*}
\operatorname{Im} \langle \sigma_j u, \sum_k \sigma_k (\partial_k - i A_k) u \rangle 
& =  \operatorname{Im} \langle u , \sum_k \sigma_j \sigma_k (\partial_k - i A_k) u \rangle \\
& =  \operatorname{Im} \langle u ,  (\partial_j - i A_j) u \rangle + \sum_{k,l}  \operatorname{Im} \langle u , \epsilon_{jkl} i \sigma_l (\partial_k - i A_k) u \rangle.
\end{align*}
Furthermore,
\begin{align*}
\sum_{k,l}  \operatorname{Im} \langle u , \epsilon_{jkl} i \sigma_l (\partial_k - i A_k) u \rangle & = \sum_{k,l}  \operatorname{Re} \langle u , \epsilon_{jkl}  \sigma_l (\partial_k - i A_k) u \rangle = \sum_{k,l} \epsilon_{jkl}  \operatorname{Re} \langle u , \sigma_l \partial_k u \rangle \\
& = \frac{1}{2} \sum \epsilon_{jkl} \partial_k  \langle u , \sigma_l u \rangle = \frac{1}{2} \left[ \nabla \times \langle u , \sigma u \rangle \right]_j.
\end{align*}

\section{The energy structure}
\label{sec:preliminaries}

\subsection{Conservation laws}

The equations conserve charge and energy.

\begin{lemma}
Let $u$ be a solution of the Pauli-Darwin equation \eqref{eq:PD_Pauli}-\eqref{eq:PD_data} or the Pauli-Poisswell equation \eqref{eq:PPw_Pauli}-\eqref{eq:PPW_data} which is smooth and decays rapidly in space. Then the continuity equation
\begin{equation} 
    \partial_t |u|^2 +  \text{\emph{div}}J = 0,
\end{equation}
holds, which implies the conservation of the charge 
\begin{equation}
\label{eq:charge}
    Q(t) = \|u(t)\|_{L^2}^2.
\end{equation}
Furthermore, the energy
\begin{align}
\label{eq:energy}
    E(t) = \|(\sigma\cdot \nabla_A )u(t)\|_{L^2}^2 + \|\nabla A(t)\|_{L^2}^2 + \|\nabla V(t)\|_{L^2}^2
\end{align}
is conserved.
\label{th:energy_lemma}
\begin{proof}
The Pauli current \eqref{eq:PD_current} can be rewritten as
\begin{equation}
    J(u,A) = \Im\langle{u},\nabla u\rangle -A|u|^2 + \frac{1}{2}\nabla \times \langle{u},\sigma u\rangle.
\end{equation}
Now $\text{div}\Im \langle{u},\nabla u\rangle =\Im \langle{u},\Delta u\rangle $ and $\text{div} (A|u|^2)  = (\text{div}A)|u|^2  + 2A \cdot \Re\langle{u},\nabla u\rangle $.
From this and $\text{div }\text{curl} =0 $ we obtain 
\begin{align*}
    \text{div} J &= \Re(-i \langle{u}, \Delta u\rangle) - 2A \cdot \Re \langle{u},\nabla u\rangle - (\text{div}A)|u|^2 \\
    &= \Re\langle u, -i\Delta u - 2A\cdot \nabla u -(\text{div}A)u\rangle \\
    &= -2\Re\langle u , \partial_t u\rangle + \Re \langle u , -2 iVu + i |A|^2u + i (\sigma \cdot B) u \rangle
\end{align*}
The last term equals zero (note that $\langle u,(\sigma \cdot B) u\rangle \in \mathbb{R}$), yielding the continuity equation. Integrating the continuity equation in space gives the conservation of charge.

\medskip

To prove the conservation of energy for the Pauli-Poisswell equation, 
we use the equation $\partial_t u +iVu = (i/2)(\sigma \cdot \nabla_A)^2u$ and the identity $J = \Im \langle \sigma u, \sigma \cdot \nabla_A u \rangle$ to obtain

\begin{align}
         \frac{d}{dt} \int  |(\sigma \cdot \nabla_A) u|^2\dd x =& 2\Re \int \langle (\partial_t+iV)(\sigma \cdot \nabla_{A})u, (\sigma \cdot \nabla_{A}) u \rangle \dd x, \nonumber \\
          =& 2\Re \int \langle (\sigma \cdot \nabla_{A})(\partial_t+iV)u, (\sigma \cdot \nabla_{A}) u \rangle \dd x \\
         & \qquad + 2\Re \int \langle i(\sigma \cdot (-\nabla V-\partial_t A))u, (\sigma \cdot  \nabla_{A})  u \rangle \dd x, \nonumber\\ 
         =& \Re i \int \langle {}(\sigma \cdot \nabla_A)^2 u,(\sigma \cdot \nabla_A)^2 u \rangle \dd x\nonumber\\
         & \qquad +2\Im \int  (-\nabla V-\partial_t A)\cdot \langle \sigma u, (\sigma \cdot \nabla_A) u\rangle \dd x \\
=& -2\int \nabla V \cdot J \dd x - 2\int \partial_t A \cdot J \dd x 
\end{align}
We conclude by noting that
\begin{equation*}
    -2\int \nabla V \cdot J \dd x = 2\int V \text{div}J \dd x = -2\int V \partial_t |u|^2 \dd x  = 2\int V \partial_t \Delta V\dd x  = -  \frac{d}{dt} \int |\nabla V|^2 \dd x ,
\end{equation*}
where we used the continuity equation and 
\begin{equation}
    -2\int \partial_t A \cdot J \dd x = 2\int \partial_t A \cdot \Delta A\dd x  = - \frac{d}{dt} \int |\nabla A|^2 \dd x
\label{eq:energy pauli poisswell}
\end{equation}
where we used the Poisson equations $-\Delta V = |u|^2$ and $-\Delta A = -J$ in the second step.

For the Pauli-Darwin equation, $J$ has to be replaced by the transversal current $J-\partial_t \nabla V$ in the above equation. But due to the Coulomb gauge, which holds for the Pauli-Darwin equation,
\begin{equation}
    \int \partial_t A \cdot \partial_t \nabla V \dd x = \int \partial_t (\nabla \cdot A) \partial_t V \dd x = 0.
    \label{eq:energy pauli darwin}
\end{equation}
\end{proof}
\end{lemma}

\subsection{Parabolic regularization and dissipation of energy}
The existence of global weak solutions is proved by regularizing the Pauli-Darwin and Poisswell equations in a way that is compatible with the energy structure: the energy is not conserved anymore, but it is decreasing with time. We define the regularized equations as
\begin{align}\label{eq:new regularization}
\partial_{t} u &= - (i+\epsilon) H u  + \epsilon \frac{(u,Hu)_{L^2}}{\|u_0\|_{L^2}^2} u, \\
\quad u(0,x) &=u_0(x),
\end{align}
for $\epsilon>0$ with 
\begin{equation}
    \label{hamiltonian}
H = -\frac{1}{2}(\sigma \cdot \nabla_{A})^2 + V
\end{equation}
and $-\Delta V = |u|^{2}$. For the Pauli-Darwin equation the magnetic field is given by
\begin{equation}
    -\Delta A = J- \partial_t \nabla V,  \quad \nabla \cdot A = 0,
\end{equation}
with the transversal current $J- \partial_t \nabla V$ and Coulomb gauge. For the Pauli-Poisswell equation it is given by
\begin{equation}
    -\Delta A = J,  \quad \nabla \cdot A +\partial_t V = 0,
\end{equation}
with the full current $J$ and the Lorenz gauge. Notice that
\begin{align}
(u,Hu)_{L^{2}} = \int \langle u , \left[ -\frac{1}{2}(\sigma \cdot \nabla_A)^2u +V \right] u\rangle \dd x &= \frac{1}{2}\int |(\sigma \cdot \nabla_A)u|^2 \dd x + \int V |u|^2 \dd x \nonumber \\
&= \frac{1}{2}\|(\sigma \cdot \nabla_A)u\|_{L^2}^2 + \| \nabla V\|_{L^2}^2 \label{eq: regularizing term new}
\end{align}
where we used $-\Delta V = |u|^2$ and integration by parts. The following lemma ensures that the energy 
\begin{equation}
    E(t) = \|(\sigma \cdot \nabla_A)u(t)\|_{L^2}^2 + \|\nabla V(t)\|^2_{L^2} + \|\nabla A(t)\|^2_{L^2}
\end{equation}
of the regularized system decays. 
\begin{lemma}
\label{thm:decay lemma}Let $u$ be a smooth and rapidly decaying solution of the regularized equation \eqref{eq:new regularization} with initial data $u(0,x) = u_0(x)$. Then $\|u(t)\|_{L^2}=\|u_0\|_{L^2}$ for all $t$ and the energy $E(t)$ satisfies
\begin{equation}
    \frac{\dd}{\dd t} E(t) = - 4\epsilon (\|Hu(t)\|_{L^2}^2 - \frac{(u(t),Hu(t))^2}{\|u(t)\|_{L^2}^2}) \leq 0
\end{equation}
for $\epsilon > 0$. 
\end{lemma}

\begin{proof} Take the equation
\begin{align}\label{eq:new regularization'}
\partial_{t} u &= - (i+\epsilon) H u  + \epsilon \frac{(u,Hu)_{L^2}}{\|u(t)\|_{L^2}^2} u, \\
\quad u(0,x) &=u_0(x),
\end{align} and taking the time derivative of $\|u(t)\|_{L^2}^2$,
\begin{align*}
\frac{\dd}{\dd t}\|u\|^2_{L^2} &= 2 \Re (u,\partial_t u ) = 2 \Re (u,- (i+\epsilon) H u  + \epsilon \frac{(u,Hu)}{\|u\|_{L^2}^2} u) \\
&= -2 \epsilon \Re(u, Hu) + 2\epsilon \frac{(u,u)}{\|u\|_{L^2}^2}\Re(u,Hu) \\
&=0
\end{align*}
Therefore we can take $\|u_0\|_2^2$ instead of $\|u(t)\|_2^2$ in the denominator in \eqref{eq:new regularization'}. Turning to the energy, we obtain, using $-\Delta A = J$ for the Pauli-Poisswell equation,
\begin{align}
\frac{\dd}{\dd t} E(t) =& \frac{\dd}{ \dd t} \left( \|(\sigma \cdot \nabla_A)u\|_{L^2}^2 + \|\nabla V\|^2_{L^2} + \|\nabla A\|^2_{L^2}\right)  \nonumber\\
=& 2\Re((\sigma \cdot \nabla_A)u,(\sigma \cdot \nabla_A)\partial_t u) + 2\Re((\sigma \cdot \nabla_A)u,-i(\sigma \cdot \partial_t A) u) \nonumber\\
& \qquad \qquad + 4\Re ( Vu, \partial_t u) +  2\int \partial_t A \cdot J \\
=& 4\Re(-\frac{1}{2}(\sigma \cdot \nabla_A)^2u +Vu,\partial_t u) + 2\Re((\sigma \cdot \nabla_A)u,-i(\sigma \cdot \partial_t A) u) +  2\int \partial_t A \cdot J \nonumber\\
=& 4\Re(Hu,\partial_t u)+ 2\Re((\sigma \cdot \nabla_A)u,-i(\sigma \cdot \partial_t A) u) +  2\int \partial_t A \cdot J .\label{eq:regularized energy 1}
\end{align}
We claim that the last two terms cancel. Indeed,
\begin{align*}
\partial_t A \cdot J= \Im \langle (\sigma \cdot\partial_t A) u , (\sigma \cdot \nabla_A)u \rangle = -\Re  \langle (\sigma \cdot \nabla_A)u  , -i(\sigma \cdot\partial_t A) u \rangle 
\end{align*}
For the Pauli-Darwin equation, $-\Delta A = J- \partial_t \nabla V$ and $\nabla \cdot A = 0$ and one can proceed as in \eqref{eq:energy pauli poisswell}-\eqref{eq:energy pauli darwin},{i.e.
\begin{equation}
    \frac{d}{dt}\int|\nabla A|^2 = -2\int  \partial_t A\cdot \Delta A = 2\int \partial_t A \cdot (J-\partial_t \nabla V) = 2\int \partial_t A \cdot J + \int \partial_t (\nabla \cdot A)\partial_t V = 0
\end{equation}
}

The first term in \eqref{eq:regularized energy 1} becomes
\begin{align*}
4\Re(Hu,\partial_t u) &= 4\Re(Hu,- (i+\epsilon) H u  + \epsilon \frac{(u,Hu)}{\|u\|_{L^2}^2} u)  \\ &= -4\epsilon\Re(Hu, H u)  + 4\epsilon\Re (Hu, \frac{(u,Hu)}{\|u\|_{L^2}^2} u) \\
&=- 4\epsilon (\|Hu\|_{L^2}^2 - \frac{(u,Hu)^2}{\|u\|_{L^2}^2})
\end{align*}
Thus,
\begin{align*}
\frac{\dd E}{\dd t}  = - 4\epsilon \left(\|Hu\|_{L^2}^2 - \frac{(u,Hu)^2}{\|u\|_{L^2}^2}\right) \leq 0
\end{align*}
by the Cauchy-Schwarz inequality.
\end{proof}

\begin{remark}
    The regularization \eqref{eq:new regularization} was used in \cite{kieffer2020time}. In fact, the regularization in \cite{guo1995global} cannot be used in a straightforward manner, as the the additional spin term $\sigma \cdot B$ does not allow a simple argument for the decay of the energy as in \cite{guo1995global}.
\end{remark}

\section{Elliptic estimates} 
\label{Elliptic estimates}

In this section, we prove estimates on the solutions of the Poisson equations satisfied by the potentials $A,V$.

\begin{lemma}
\label{lemma V}
Let $V=(-\Delta)^{-1}|u|^2$ and $s\geq 1$. Then,
\begin{align*}
& \| V \|_{H^s} \lesssim \| u \|_{H^s}^2 \\
& \|V u\|_{H^s} \lesssim \|u\|_{H^1}^2 \|u\|_{H^s}
\end{align*}
\end{lemma}

\begin{proof}To prove the first assertion, we write
$$
\| \Delta^{-1} |u|^2 \|_{H^s} \lesssim \| \Delta^{-1} |u|^2 \|_{L^2} + \| |D|^{s-2} |u|^2 \|_{L^2} \lesssim \| |u|^2 \|
$$

By the fractional Leibniz rule \cite{grafakos2014kato}, 
\begin{equation*}
    \|((-\Delta)^{-1}|u|^2)u\|_{H^s} \lesssim \|(-\Delta)^{-1}|u|^2\|_{L^\infty} \|u\|_{H^s} + \|\langle D \rangle^s (-\Delta)^{-1}|u|^2\|_{L^6} \|u\|_{L^3}.
\end{equation*}

We now examine the terms on the above right-hand. On the one hand, by classical estimates on the Riesz potential and the Sobolev embedding theorem,
\begin{equation*}
\| (-\Delta)^{-1}|u|^2\|_{L^\infty} \lesssim \| |u|^2 \|_{L^{3/2+}} + \| |u|^2 \|_{L^{3/2-}} = \| u \|_{L^{3+}}^2 + \| u \|_{L^{3-}}^2 \lesssim \| u \|_{H^1}^2.
\end{equation*}

On the other hand, by splitting between high and low frequencies, 
\begin{equation*}
\|\langle D \rangle^s (-\Delta)^{-1}|u|^2\|_{L^6} 
\lesssim \| (-\Delta)^{-1}|u|^2\|_{L^6} + \| \langle D \rangle^{s-2} |u|^2\|_{L^6}.
\end{equation*}
The first term on the right-hand side can be bounded with the help of H\"older's inequality and the Sobolev embedding theorem as follows
$$
\| (-\Delta)^{-1}|u|^2\|_{L^6} \lesssim \| |u|^2 \|_{L^{6/5}} \lesssim \| u \|_{L^{12/5}}^2 \lesssim \| u \|_{H^1}^2.
$$
For the second term on the right-hand side, we need to distinguish between two cases: if $1 \leq s \leq 2$,
$$
 \| \langle D \rangle^{s-2} |u|^2\|_{L^6} \lesssim \| |u|^2\|_{L^6} \lesssim \| u \|_{L^3}^2 \lesssim \| u \|_{H^1}^2.
$$
If $s \geq 2$, the same arguments as above give
$$
\| \langle D \rangle^{s-2} |u|^2\|_{L^6} \lesssim \| \langle D \rangle^{s-2} u \|_{L^{\infty}} \| u \|_{L^{6}} \lesssim \| \langle D \rangle^{s} u \|_{L^2} \| u \|_{L^6} \lesssim \| u \|_{H^1} \| u \|_{H^s}.
$$
Gathering the above estimates gives the desired result.
\end{proof}

\begin{lemma}
\label{lem:apriori_magnetic}
\begin{itemize}
\item[(i)] (Existence and uniqueness of a solution) 
Let $u\in H^1$ and consider the boundary value problems on $\mathbb{R}^3$
\begin{align}
    -\Delta A = J,  && |A(x)| \rightarrow \infty \text{ as } |x|\rightarrow 0
\end{align}
or
\begin{align}
 -\Delta A = \mathbb{P}J, && |A(x)| \rightarrow \infty \text{ as } |x|\rightarrow 0
\end{align}
where $\mathbb{P}$ denotes the Leray projection and
\begin{equation}
    J = \Im \langle u, \nabla_A u\rangle +\frac{1}{2} \nabla \times \langle u,\sigma u\rangle.
\end{equation}
These problems have a unique solution $A \in \dot{H}^1$. 

\item[(ii)] ($L^2$ estimates)
Let $A,A'$ be two solutions corresponding to $J,J'$ defined by $J=J(u,A)$ and $J'=(u',A')$. Then $A,A'$ satisfy the estimates
\begin{align}
 \|\nabla A\|_{L^2} &\lesssim \|u\|_{H^1},\label{eq:estimate A H1} \\
\|\nabla(A-A')\|_{L^2} &\lesssim F(\| u \|_{H^1}, \| u' \|_{H^1}) \|u-u'\|_{H^1} \label{eq: estimate nabla A - A'}\\ 
\|\nabla A\|_{{H}^s}  & \lesssim F(\|u\|_{{H}^s}).
        \label{eq:estimate A Hs}
\end{align}
where $s > \frac 32$ and $F$ always stands for a continuous function vanishing at the origin. 

\item[(iii)] ($L^\infty$ estimates) 
If $s>\frac 32$, then
\begin{equation}
  \| A\|_{L^\infty}  + \|\nabla A\|_{L^\infty} \lesssim F(\| u \|_{H^s}) .
    \label{eq:estimate nabla A infty}
\end{equation}
\end{itemize}

\begin{proof}
\noindent (i) The problem $-\Delta A = J$ can be written as
\begin{align}
    LA(x) &= f(x), \quad x \in \mathbb{R}^3, \quad |A(x)| \rightarrow 0 \text{ as }|x|\rightarrow \infty \\
    f(x) &= -\Im \langle u(x), \nabla u(x)\rangle - \frac{1}{2} \nabla \times \langle u(x), \sigma u(x) \rangle
\end{align}
where $L$ is defined as
\begin{equation*}
    L = \Delta - |u|^2.
\end{equation*}

Defining the bilinear form $B$ as
\begin{equation}
    B[A,A'] = \int \nabla A \cdot \nabla A' \dd x + \int A\cdot A'|u|^2 \dd x,
\end{equation}
it is bounded and coercive on $\dot{H}^1(\mathbb{R}^3)$ if we can bound
$$
\int A \cdot A' |u|^2 \dd x \lesssim \| \nabla A \|_{L^2} \| \nabla A' \|_{L^2}, 
$$
which is the case if $u \in L^3 \subset H^1$ by Sobolev embedding. As for the form $(f,\cdot)$, it is bounded on $\dot{H}^1(\mathbb{R}^3)$ if $f \in L^{6/5} \subset \dot{H}^{-1}$, which is also ensured by $u \in H^1$, as follows from H\"older's inequality and the Sobolev embedding theorem.

Therefore, the Lax-Milgram theorem (or the Riesz representation theorem) gives a unique solution $A \in \dot H^1$ if $u \in H^1$.

In the case of Pauli-Darwin ($-\Delta A = \mathbb{P} J$), the operator $L$ needs to be modified to $L = \Delta - \mathbb{P} (|u|^2 \cdot)$, and the function $f$ to $\mathbb{P} f$. However, the same argument as above applies almost verbatim since $\mathbb{P}$ is bounded on $L^p$ spaces and $A$ is restricted to the space of divergence-free fields.

\medskip

\noindent (ii) Using boundedness of $\mathbb{P}$ on Lebesgue spaces, it suffices to deal with the case of Pauli-Poisswell $-\Delta A = J$.
Multiplying both sides of the equation with $A$ and integrating over $\mathbb{R}^3$,
\begin{align*}
   \int_{\mathbb{R}^3} \left[ |A|^2|u|^2 + |\nabla A|^2 \right] \dd x = \int_{\mathbb{R}^3} A\cdot (\Im \langle{u},\nabla u\rangle { +\frac{1}{2}}\nabla \times \langle u, \sigma u\rangle) \dd x.
\end{align*}
From
\begin{align*}
    & \|\nabla A\|_{L^2}^2 + \|Au\|_{L^2}^2 \leq \|Au\|_{L^2} \|\nabla u\|_{L^2}
\end{align*}we deduce
$$
\| \nabla A\|_{L^2} \lesssim \| u \|_{H^1}.
$$
For \eqref{eq: estimate nabla A - A'}, subtract the equations for $A$ and $A'$, multiply by $A-A'$ and integrate over $\mathbb{R}^3$ to obtain
\begin{equation}
\begin{split}
& \int|\nabla(A-A')|^2 \dd x + \int |A-A'|^2|u'|^2 \dd x \\
& \qquad = \int (A-A') \cdot \left[ \Im \langle{u},\nabla u\rangle - \Im \langle{u'},\nabla u'\rangle + \frac 12 \nabla \times \langle u,\sigma u \rangle \right.\\
& \qquad \qquad \qquad \qquad - \left. \frac 12 \nabla \times \langle u',\sigma u' \rangle + A (|u'|^2 - |u|^2) \right] \dd x
\label{eq:estimate difference A - A'}
\end{split}
\end{equation}
Bounding the right-hand side with the help of H\"older's inequality and the Sobolev embedding theorem gives
\begin{align*}
& \int|\nabla(A-A')|^2 \dd x + \int |A-A'|^2|u'|^2 \dd x \\
& \qquad \lesssim \| \nabla(A-A') \|_{L^2} \| u \|_{H^1} \| u-u' \|_{H^1} + \| \nabla(A-A') \|_{L^2} \| \nabla A \|_{L^2}\| u-u' \|_{H^1},
\end{align*}
which leads to the desired inequality.

Turning to \eqref{eq:estimate A Hs} in the case $s>2$, we have by the Poisson equation on $A$
\begin{equation}
\label{estimateA}
\| \nabla A \|_{H^s} = \| \nabla \Delta^{-1} J \|_{H^s} \sim \| D^{-1} J \|_{L^2} + \| D^{s-1} J \|_{L^2}.
\end{equation}
To estimate the first term on the right-hand side, we use successively the Sobolev embedding theorem, the definition of $J$, H\"older's inequality, and once again the Sobolev embedding theorem to obtain that
$$
\| D^{-1} J \|_{L^2} \lesssim \| J \|_{L^{6/5}} \lesssim \| u \|_{L^3} \| \nabla u \|_{L^2} + \| A \|_{L^\infty} \| u \|_{L^2} \| u \|_{L^3} \lesssim F( \| u \|_{H^s}).
$$
For the second term on the right-hand side of \eqref{estimateA}, we use successively the the definition of $J$, the fractional Leibniz rule, H\"older's inequality and the Sobolev embedding theorem to get
\begin{align*}
\| D^{s-1} J \|_{L^2} & \lesssim \| D^{s-1} u \|_{L^6} \| \nabla u \|_{L^3} + \| u \|_{L^\infty} \| D^s u \|_{L^2} \\
& \qquad \qquad + \| A \|_{L^\infty} \| D^{s-1} u \|_{L^2} \| u \|_{L^\infty} + \| D^{s-1} A \|_{L^6} \| u \|_{L^6}^2 \\
& \lesssim F(\| u \|_{H^s}) + \| u \|_{H^s}^2 \| D^{s} A \|_{L^2}
\end{align*}
Overall, we proved that
$$
\| \nabla A \|_{H^s} \lesssim F(\| u \|_{H^s}) + \| u \|_{H^s}^2 \| D^{s} A \|_{L^2}.
$$
To conclude, it suffices to use an interpolation inequality followed by Young's inequality which gives
\begin{align*}
\| u \|_{H^s}^2 \| D^{s} A \|_{L^2} & \lesssim \| u \|_{H^s}^2 \| D A \|_{L^2}^{\frac 1 {s}} \| D^{s+1} A \|_{L^2}  ^{1 - \frac 1 {s}} \\
& \lesssim F(\| u \|_{H^s}) \| \nabla A \|_{H^s}^{1 - \frac 1 {s}}\\
& \lesssim c^{1-s} F(\| u \|_{H^s}) +  c \| \nabla A \|_{H^s}. 
\end{align*}
(where the function $F$ differs in the last two lines above). This means that
$$
\| \nabla A \|_{H^s} \lesssim F(\| u \|_{H^s}) + c^{1-s} F(\| u \|_{H^s}) +  c \| \nabla A \|_{H^s}.
$$
Choosing $c$ sufficiently small gives the desired result.
\medskip

\noindent (iii) 
We only prove the bound on $\| \nabla A \|_{L^\infty}$ since the bound on $\| A \|_{L^\infty}$ is easier. By definition of $A$,
$$
\nabla A = \nabla \Delta^{-1} \left(\Im \langle u, \nabla_A u\rangle + \frac{1}{2} \nabla \times \langle u,\sigma u\rangle \right),
$$
so that, by mapping properties of the Riesz potential,
$$
\| \nabla A \|_{L^\infty} \lesssim \left\| A u^2 \right\|_{L^{3-}} +  \left\| A u^2 \right\|_{L^{3+}} + \left\| u \nabla u \right\|_{L^{3-}} +  \left\| u \nabla u \right\|_{L^{3+}}.
$$
By H\"older's inequality, we can bound
\begin{align*}
\| \nabla A \|_{L^\infty} \lesssim \left\| A \right\|_{L^{6}} \| u \|_{L^{12-}}^2 + \left\| A \right\|_{L^{6}} \| u \|_{L^{12+}}^2 + \| \nabla u \|_{L^{3-}} \| u \|_{L^\infty} + \| \nabla u \|_{L^{3+}} \| u \|_{L^\infty}.
\end{align*}
Using that $\| \nabla A \|_{L^2} \lesssim \| u \|_{H^1}^2$ and the Sobolev embedding theorem, this implies that
$$
\| \nabla A \|_{L^\infty} \lesssim (1+ \|u\|_{H^1}^2) \|u\|^2_{H^{3/2+}} ,
$$
which was the desired estimate.
\end{proof}
\end{lemma}

In order to pass to the weak limit in $H^1$ we need strong convergence of the vector potential $A$, depending on $u$. To this end we establish the following lemma.
\begin{lemma}\label{th:strong_convergence_A} For $\chi$ a cutoff function supported on $B(0,2)$ and equal to $1$ on $B(0,1)$, we denote $\chi_R = \chi \left( \frac \cdot R \right)$. Assume that $u,u' \in H^1$ and $A,A' \in \dot H^1$ solve
\begin{align*}
& -\Delta A = \Im\langle{u}, \nabla_A u\rangle+\frac{1}{2}\nabla \times \langle{u},\sigma u\rangle \\
& -\Delta A' = \Im\langle{u'}, \nabla_{A'} u'\rangle+\frac{1}{2}\nabla \times \langle{u'},\sigma u'\rangle.
\end{align*}
Then
\begin{align*}
\| A - A' \|_{L^6(B(0,R))} & \lesssim  \|  u -u' \|_{L^3 (B(0,20R))}\left[\| u\|_{H^1} + \| u' \|_{H^1} \| \right] \left[ 1 + \| \nabla A' \|_{L^2} \right] \\
& \qquad \qquad + \frac{1}{R^{1/4}}  (\| u \|_{ H^1} + \| u' \|_{ H^1})^2 (1 +  \|\nabla A'\|_{L^2})
\end{align*}
\end{lemma}

\begin{proof}
The equations satisfied by $A$ and $A'$ imply that
\begin{align*}
(-\Delta + |u|^2) (A-A') & = \underbrace{\langle{u}, \nabla u\rangle - \langle{u'}, \nabla u'\rangle}_{\displaystyle F_1} + \underbrace{\frac{1}{2}\nabla \times \langle{u},\sigma u\rangle - \frac{1}{2}\nabla \times \langle{u'},\sigma u'\rangle}_{\displaystyle  F_2} + \underbrace{(|u'|^2 - |u|^2) A'}_{\displaystyle  F_3}.
\end{align*}
The idea is now to split the right-hand side
$$
F=F_1 + F_2 + F_3
$$
into a 'close' and a 'far' contribution by writing
$$
F = \chi_{10R} F + (1- \chi_{10R}) F.
$$

\medskip

\noindent \underline{The close contribution.} We estimate here the solution $A_{c}$ of
$$
(-\Delta + |u|^2) A_{c} = \chi_{10R} F 
$$
by an energy estimate: taking the $L^2$ inner product of this equation with $A_c$ gives
$$
\int \left( |\nabla A_c|^2 + |u|^2 |A_c|^2 \right) \dd x
= \int  \chi_{10R} F \cdot A_c \dd x.
$$
We start by estimating
$$
\int \chi_{10R} F_1 \cdot A_c \dd x =
\int \chi_{10R} \left[ \langle u -u',\nabla u \rangle + \langle u',\nabla u - \nabla u' \rangle \right]  A_c \dd x.
$$
For the first term on the right-hand side, H\"older's inequality and the Sobolev embedding theorem suffice
\begin{align*}
\left| \int \chi_{10R} \langle u -u',\nabla u \rangle \cdot A_n \dd x \right| & \lesssim  \|  u -u' \|_{L^3 (B(0,20R))} \| \nabla u \|_{L^2} \| A_c \|_{L^6} \\
& \lesssim \|  u -u' \|_{L^3 (B(0,20R))} \| u \|_{H^1} \| \nabla A_c \|_{L^2}.
\end{align*}
For the second term, we need to integrate by parts
\begin{align*}
& \int \chi_{10R} \langle u',\nabla u - \nabla u' \rangle  A_c \dd x \\
& = - \int \nabla \chi_{10R} \langle u', u - u' \rangle  A_c \dd x
- \int  \chi_{10R} \langle \nabla u', u - u' \rangle  A_c \dd x
-  \int  \chi_{10R} \langle  u', u - u' \rangle \nabla  A_c \dd x
\end{align*}
before applying H\"older's inequality and the Sobolev embedding theorem to obtain
$$
\left|  \int \chi_{10R} \langle u',\nabla u - \nabla u' \rangle  A_c \dd x \right| \lesssim \|  u -u' \|_{L^3 (B(0,20R))} \| u' \|_{H^1} \| \nabla A_c \|_{L^2}.
$$
The terms $F_2$ and $F_3$ can be estimated by the same arguments to yield
\begin{align*}
& \left| \int \chi_{10R} F_2 \cdot A_c \dd x \right| \lesssim \|  u -u' \|_{L^3 (B(0,20R))} \| u' \|_{H^1} \| \nabla A_c \|_{L^2} \\
& \left| \int \chi_{10R} F_3 \cdot A_c \dd x \right| \lesssim \|  u -u' \|_{L^3 (B(0,20R))} \| u' \|_{H^1} \| \nabla A_c \|_{L^2} \| \nabla A' \|_{L^2}.
\end{align*}
Overall, we find that
\begin{align*}
\| \nabla A_c \|_{L^2}^2 & \leq \left|\int \chi_{10R} F \cdot A_c \dd x \right| \\
& \lesssim  \|  u -u' \|_{L^3 (B(0,20R))}\left[\| u\|_{H^1} + \| u' \|_{H^1} \| \right] \left[ 1 + \| \nabla A' \|_{L^2} \right] \| \nabla A_c \|_{L^2},
\end{align*}
which implies that
$$
\| \nabla A_c \|_{L^2} \lesssim \|  u -u' \|_{L^3 (B(0,20R))}\left[\| u\|_{H^1} + \| u' \|_{H^1} \| \right] \left[ 1 + \| \nabla A' \|_{L^2} \right].
$$
The same bound for $\| A_c \|_{L^6}$ follows by yet another application of the Sobolev embedding theorem.

\medskip
\noindent \underline{The far contribution.}  We estimate here the solution $A_{f}$ of
$$
(-\Delta + |u|^2) A_{f} = (1- \chi_{10R}) F 
$$
by considering the fundamental solution $K(x,y)$ of $-\Delta + |u|^2$. We claim first that it satisfies
\begin{equation}
\label{boundK}
| K(x,y)| \lesssim \frac{1}{|x-y|}.
\end{equation}
The justification of this inequality relies on the convolution kernel of the fundamental solution of the 'flat' Laplacian $K_0(x) = \frac{C_0}{|x|}$. By the maximum principle \cite{Evans}, the solutions $A$ and $A_0$ of $(-\Delta + |u|^2) A = f$ and $-\Delta A_0 = f$ are $\geq 0$ if $f \geq 0$. This implies first that $K(x,y) \geq 0$. Furthermore, $A \leq A_0$ since $-\Delta(A-A_0) = -|u|^2 A \leq 0$, which implies that $K(x,y) \leq K_0(x-y)$. Combining these two inequalities on $K$ gives the bound \eqref{boundK}.

With the help of the fundamental solution $K$, we can now write for $R>1$
$$
\chi_R(x) A_f(x) = \chi_R(x) \int K(x,y) (1- \chi_{10R}(y)) F(y) \dd y.
$$
Due to \eqref{boundK},
$$
\chi_R(x) K(x,y) [1-\chi_{10 R}(y)] \lesssim \frac{1}{|x-y|^{3/4}} \frac{1}{R^{1/4}}.
$$
Furthermore, by H\"older's inequality and the Sobolev embedding theorem,
$$
\| F \|_{L^{12/11}} \lesssim (\| u \|_{ H^1} + \| u' \|_{ H^1})^2 (1 +  \|\nabla A'\|_{L^2})
$$
Combining the two above inequalities and applying the Hardy-Littlewood-Sobolev inequality gives that
$$
\| \chi_R A_f \|_{L^6} \lesssim \frac{1}{R^{1/4}}  (\| u \|_{H^1} + \| u' \|_{H^1})^2 (1 +  \|\nabla A'\|_{L^2}).
$$
\end{proof}

Next we collect some estimates for the regularizing term.

\begin{lemma}
\label{thm:estimates energy}
Let $s\geq 1$ and let $u,u'\in H^1$. Then the following estimates hold:
\begin{align*}
|(u,Hu)| & \leq F(\|u\|_{H^1}) \\
\|(u,H u)u - (u',H'u')u'\|_{H^s} & \lesssim F(\|u\|_{H^1},\|u'\|_{H^1},\|u\|_{H^s})\|u-u'\|_{H^s} \\
\left\|\frac{(u,H u)}{\|u_0\|_{L^2}^2}u - \frac{(u',H'u')}{\|u_0'\|_{L^2}^2}u'\right\|_{H^s} &  \lesssim \frac{F(\|u\|_{H^1},\|u'\|_{H^1},\|u\|_{H^s})}{\|u_0\|_{L^2}^2 |u_0'\|_{L^2}^2}\|u-u'\|_{H^s}.
\end{align*}
Here, $H,H'$ denote the Hamiltonian \eqref{hamiltonian} with respect to $A,A'$ and $V,V'$, defined by 
\begin{align*}
-\Delta A &= J(u,A) & -\Delta V &= |u|^2 \\ 
-\Delta A' &= J(u',A') & -\Delta V' &= |u'|^2.
\end{align*} 
The same estimates hold if $J,J'$ are replaced by $\mathbb{P}J,\mathbb{P}J'$.
\end{lemma}

\begin{proof}
The first estimate follows from lemmas \ref{lemma V} and \ref{lem:apriori_magnetic}.

For the estimates on differences, we write
\begin{align}
 \left|(u,H u) - (u',H'u')\right| &\leq \left|\|(\sigma \cdot \nabla_A)u\|^2_{L^2} -  \|(\sigma \cdot \nabla_{A'})u'\|^2_{L^2}\right| \nonumber \\ &\qquad +  \left|\|\nabla V\|^2_{L^2}  - \|\nabla V'\|^2_{L^2} \right|.
    \label{eq:difference kinetic energies}
\end{align}
The first term on the RHS can be bounded by
\begin{align*}
    \|(\sigma \cdot \nabla_A)u\|^2_{L^2} -  \|(\sigma \cdot \nabla_{A'})u'\|^2_{L^2} &= ((\sigma \cdot \nabla)(u-u'), (\sigma\cdot \nabla_A) u) + (-i(A-A')u, (\sigma\cdot \nabla_A) u) \\ &+ (-iA'(u-u'),(\sigma\cdot \nabla_A) u) + ((\sigma \cdot \nabla_{A'})u', (\sigma \cdot \nabla)(u-u')) \\ &+ ((\sigma \cdot \nabla_{A'})u',-i(A-A')u) + ((\sigma \cdot \nabla_{A'})u',-iA'(u-u')).
\end{align*}
This can be estimated as
\begin{align*}
    \left|\|(\sigma \cdot \nabla_A)u\|^2_{L^2} -  \|(\sigma \cdot \nabla_{A'})u'\|^2_{L^2}\right| \lesssim& \|\nabla(u-u')\|_{L^2} \| \nabla_A u\|_{L^2} + \|A-A'\|_{L^6} \|u\|_3 \| \nabla_A u \|_{L^2} \\ &+ \|A'\|_{L^6} \|u-u'\|_3\| \nabla_A u\|_{L^2} + \| \nabla_{A'} u'\|_{L^2}\|\nabla(u-u')\|_{L^2} \\ &+ \| \nabla_{A'} u'\|_{L^2}\|A-A'\|_{L^6} \|u\|_3 + \| \nabla_{A'} u'\|_{L^2}\|A'\|_{L^6} \|u-u'\|_3.
\end{align*}
Together with the estimates \eqref{eq: estimate nabla A - A'} and $\|\nabla_{A}u\|_{L^2} \lesssim \|u\|_{H^1}$, $\|\nabla_{A'}u'\|_{L^2} \lesssim \|u'\|_{H^1}$ we obtain
\begin{equation}
    \left|\|(\sigma \cdot \nabla_A)u\|^2_{L^2} -  \|(\sigma \cdot \nabla_{A'})u'\|^2_{L^2}\right| \lesssim F(\|u\|_{H^1},\|u'\|_{H^1}) \|u-u'\|_{H^1},
\end{equation}

For the electric potential, we obtain
\begin{align*}
    \|\nabla V\|_{L^2}^2 - \|\nabla V'\|_{L^2}^2 &= \int |\nabla V|^2 - \int |\nabla V'|^2 \\
    &= \int \nabla V \cdot (\nabla V - \nabla V') + \int \nabla V' \cdot (\nabla V - \nabla V') \\
    &\leq \left(\|\nabla V\|_{L^2}+\|\nabla V'\|_{L^2}\right) \|\nabla V - \nabla V'\|_{L^2}.
\end{align*}
The last factor above can be estimated by the Sobolev embedding theorem and H\"older's inequality:
\begin{align*}
\|\nabla V - \nabla V'\|_{L^2} & = \| \Delta^{-1} \nabla (|u|^2 - |u'|^2) \|_{L^2} \lesssim \| |u|^2 - |u'|^2 \|_{L^{6/5}} \\
& \lesssim \left( \| u \|_{L^2} + \| u' \|_{L^2} \right) \| u-u' \|_{L^3} \lesssim \| u \|_{H^1} \| u - u' \|_{H^1}.
\end{align*}

Overall, we obtain
\begin{equation}
    \left|(u,H u) - (u',H'u')\right| \lesssim F(\|u\|_{H^1},\|u'\|_{H^1}) \|u-u'\|_{H^1},
\end{equation}
and thus
\begin{align}
    \left\|(u,H u)u - (u',H'u')u'\right\|_{H^s} &\leq \left|(u,H u) - (u',H'u')\right| \|u'\|_{H^s}+ |(u,Hu)|\|u-u'\|_{H^s} \nonumber \\
    &\lesssim  F(\|u\|_{H^1},\|u'\|_{H^1})\|u-u'\|_{H^1} \|u'\|_{H^s} + F(\|u\|_{H^1}) \|u-u'\|_{H^s}.
    \label{eq:difference auxilliary}
\end{align}

For the last statement we have,
\begin{align*}
    \left|\frac{(u,Hu)}{\|u_0\|_{L^2}^2}-\frac{(u',H'u')}{\|u_0'\|_{L^2}^2}\right| &\leq \left|\frac{(u,Hu)\|u_0'\|_{L^2}^2-(u',H'u')\|u_0\|_{L^2}^2}{\|u_0'\|_{L^2}^2\|u_0\|_{L^2}^2}\right| \\ &\leq \left|\frac{(u,Hu)(\|u_0\|_{L^2}^2-\|u_0'\|_{L^2}^2)}{\|u_0'\|_{L^2}^2\|u_0\|_{L^2}^2}\right| + \left|\frac{(u,Hu)-(u',H'u')}{\|u_0'\|_{L^2}^2}\right| \\
    &\leq \frac{\left|(u,Hu)\right|(\|u_0\|_{L^2}+\|u_0'\|_{L^2})}{\|u_0'\|_{L^2}^2\|u_0\|_{L^2}^2}\|u_0-u_0'\|_{L^2} + \frac{1}{\|u_0'\|_{L^2}^2}\left|(u,Hu)-(u',H'u')\right| \\
    &\leq \frac{F(\|u\|_{H^1},\|u'\|_{H^1})}{\|u_0'\|_{L^2}^2\|u_0\|_{L^2}^2} \|u-u'\|_{L^2} + \frac{F(\|u\|_{H^1},\|u'\|_{H^1})}{\|u_0'\|_{L^2}^2} \|u-u'\|_{H^1}
    \\
    &\leq \frac{F(\|u\|_{H^1},\|u'\|_{H^1})}{\|u_0'\|_{L^2}^2\|u_0\|_{L^2}^2} \|u-u'\|_{H^1}
\end{align*}
The claim follows in a similar fashion to \eqref{eq:difference auxilliary}.
\end{proof}

\section{Parabolic regularization}
\label{Parabolic regularization}
In this section we show wellposedness of the regularized equation \eqref{eq:new regularization} on a time interval depending on $\epsilon$. In Section \ref{sec:lwp}, owing to an a priori estimate (Lemma \ref{thm:grönwall}), we obtain local strong solutions in $H^s$, $s>3/2$, where the time interval of existence no longer depends on $\epsilon$, thereby proving Theorem \ref{thm:local wellposedness}.

In Section \ref{sec:weak sol}, due to the dissipation of energy (Lemma \ref{thm:decay lemma}), we can extend the regularized solutions globally in $H^1$. Thus we obtain existence of global weak solutions in the energy space when passing to the limit, thereby proving Theorem \ref{th:global weak solutions}. \\

Recall the regularized Pauli-Darwin equation 
\begin{align} \label{eq:new regularization Darwin}
 \partial_{t} u &= - (i+\epsilon) H u  + \epsilon \frac{(u,Hu)}{\|u_0\|_{L^2}^2} u, \quad H = -\frac{1}{2}(\sigma \cdot \nabla_{A})^2 + V, \\
-\Delta V &= |u|^{2}, \\
-\Delta A &= J-\partial_t \nabla V, \quad B= \nabla \times A, \\
\nabla \cdot A &= 0, \\
u(0,x) &=u_0(x). \label{eq:new regularization Darwin data}
\end{align}
and the regularized Pauli-Poisswell equation
\begin{align}
\partial_{t} u &= - (i+\epsilon) H u  + \epsilon \frac{(u,Hu)}{\|u_0\|_{L^2}^2} u, \quad H = -\frac{1}{2}(\sigma \cdot \nabla_{A})^2 + V, \label{eq:new regularization Poisswell}\\
-\Delta V &= |u|^{2}, \\
-\Delta A &= J, \quad B= \nabla \times A, \\
\nabla \cdot A + \partial_t V &= 0, \\
u(0,x) &=u_0(x).
\label{eq:new regularization Poisswell data}
\end{align}
We rewrite the equations as
\begin{align}
\partial_t u  &= \frac{i + \epsilon}{2}\Delta u + K^{\star}_{\epsilon}(u) + \epsilon\frac{(u,Hu)}{\|u_0\|_{L^2}^2} u
\label{eq:regDarwin}\\
    u(0,x) &= u_0(x)
\end{align}
where $\star = \text{D, PW}$ (Pauli-Darwin and Pauli-Poisswell respectively) and
\begin{align}
& K^{\text{D}}_{\epsilon}(u) = - (i+\epsilon)\left(i\nabla(A u) + \frac{1}{2}|A|^2 u  - \frac{1}{2} (\sigma \cdot B)u + Vu\right)\label{eq:nonlinearity Darwin} \\
& K^{\text{PW}}_{\epsilon}(u) = - (i+\epsilon)\left(i \nabla(A u) - \frac{i}{2}(\nabla \cdot A) u + \frac{1}{2}|A|^2 u  - \frac{1}{2} (\sigma \cdot B)u + Vu\right)\label{eq:nonlinearity Poisswell}
\end{align} 
together with the Poisson equation for $A$, respectively for Darwin and Poisswell:
\begin{align}
    -\Delta A &= \mathbb{P}J= \mathbb{P}\left[\Im \langle{u},\nabla_Au\rangle+\frac{1}{2} \nabla \times \langle{u},\sigma u\rangle\right], \quad \nabla \cdot A = 0\label{eq:regPaulicurrent Darwin} \\
    -\Delta A &= J= \Im \langle{u},\nabla_Au\rangle+\frac{1}{2} \nabla \times \langle{u},\sigma u\rangle, \quad \nabla \cdot A + \partial_t V = 0\label{eq:regPaulicurrent Poisswell} 
\end{align}

\begin{proposition}
\label{proposition lwp regularized}
Let $\epsilon>0$ and let $u_0 \in H^s(\mathbb{R}^3)$, with $s\geq 1$ a real number. Then there exists a $T_{\epsilon}>0$, depending on $\|u_0\|_2$, $\|u_0\|_{H^s}$ and $\epsilon$, such that the regularized Pauli-Darwin and Pauli-Poisswell equations are locally well-posed in the space $C((0,T_{\epsilon}),H^s(\mathbb{R}^3))$ with initial data $u_0$.
\end{proposition}

\begin{proof} \underline{Reformulation as a fixed point problem.}

We introduce the heat-Schrödinger semigroup 
\begin{equation*}
     S(t) = \exp(\frac{i+\epsilon}{2}t\Delta).
\end{equation*}
which enjoys the following estimates: for any $s\geq 0$, $1<r\leq 2 \leq \infty$ and $f\in L^r$, 
\begin{equation}
 \|S(t) f\|_{H^s} \lesssim (\epsilon t)^{-\frac{3}{2}\left(\frac{1}{r}-\frac{1}{2}\right)}\left(1+(\epsilon t)^{-\frac{s}{2}}\right) \|f\|_{L^r} 
 \label{eq:general estimate heat Schrödinger}
\end{equation}

We can now rely on Duhamel's formula to recast the problem we want to solve as a fixed point problem: we will show that the operator
\begin{equation}
    \Phi (u(t))= S(t)u_0 + \int_0^t S(t-\tau)\left[K_{\epsilon}^\star(u(\tau)) + \epsilon\frac{(u,H u)}{\|u_0\|_{L^2}^2} u(\tau)\right] \dd \tau,
    \label{eq:integral_equation}
\end{equation}
for $\star=\text{D},\text{PW}$ 
has a fixed point in $X_T^s$ given by
\begin{equation}
    X_T^s = \{u\in L^{\infty}_T H^s \colon \| u(t) \|_{L^{\infty}_T H^s} < R \quad \forall t\in [0,T)\}
\end{equation}
equipped with the metric
\begin{equation}
    d(u,u') = \|u-u'\|_{L^{\infty}_TH^s}
\end{equation}



\medskip

\noindent \underline{Bounds on $\Phi$.}
Let $s\geq 1$. First note that
\begin{equation*}
\|S(t)u_0\|_{H^s} = \|\langle D\rangle^{s} S(t) u_0 \|_{L^2} = \|S(t)\langle D\rangle^{s} u_0 \|_{L^2} \lesssim \|u_0\|_{H^s},
\end{equation*}
by \eqref{eq:general estimate heat Schrödinger}. We now turn to the Duhamel term and write
$$
K^\star_\epsilon(u) = \mathcal{K}(u) + Vu,
$$
since $Vu$ will be treated differently from the other terms.

By \eqref{eq:general estimate heat Schrödinger} with $p=2$ and $s=1$,
        \begin{align*}
            \|S(t-\tau)\mathcal{K}(u(\tau))\|_{H^s} \lesssim \frac{1}{\epsilon^{1/4}(t-\tau)^{1/4}}\left(1+\frac{1}{\epsilon^{1/2}(t-\tau)^{1/2}}\right)\|\langle D \rangle^{s-1} K(u(\tau))\|_{L^{3/2}}
        \end{align*}
By \eqref{eq:nonlinearity Darwin}-\eqref{eq:nonlinearity Poisswell},
\begin{align*}
\|\langle D \rangle^{s-1} \mathcal{K}(u)\|_{L^{3/2}} \leq \|Au\|_{W^{s,3/2}} + \|(\nabla A) u\|_{W^{s-1,3/2}} &+ \||A|^2 u\|_{W^{s-1,3/2}}.
\end{align*}
Since there is no $L^2$-estimate for $A$, we split the estimate for the first term into two cases: by H\"older's inequality
\begin{equation*}
\|Au\|_{L^{3/2}} \lesssim \|A\|_{L^6}\|u\|_{L^2} \lesssim \|u\|_{H^1}^2,
\end{equation*}
and with the help of Lemma \ref{lem:apriori_magnetic} and of the fractional Leibniz rule,
\begin{align*}
\|Au\|_{\dot{W}^{s,3/2}} &\lesssim \|A\|_{\dot{H}^s} \|u\|_{L^6} + \|A\|_{L^6} \|u\|_{\dot{H}^s} \\ &\lesssim \|u\|_{\dot{H}^1}(\|A\|_{\dot{H}^s}+  \|u\|_{\dot{H}^s}) 
\\ &\lesssim F(\|u\|_{H^s})
\end{align*}
We also split the estimate for the third term into two cases: by H\"older's inequality
\begin{equation*}
\||A|^2u\|_{L^{3/2}} \leq \|A\|_{L^6}^2 \|u\|_3^2 \lesssim \|u\|_{H^1}^3,
\end{equation*}
and, again with the help of Lemma \ref{lem:apriori_magnetic}, of the fractional Leibniz rule and of the Sobolev embedding theorem,
\begin{align*}
\||A|^2u\|_{\dot{W}^{s-1,3/2}} &\lesssim \||A|^2\|_{\dot{W}^{s-1,2}}\|u\|_{L^6} + \||A|^2\|_3 \|u\|_{\dot{W}^{s-1,3}} \\
&\lesssim \|u\|_{\dot{H}^1} \|A\|_{\dot{H}^{s-1/2}}^2 + \| u \|_{\dot H^1}^2 \|u\|_{\dot{H}^{s-1/2}} \\
&\lesssim F(\|u\|_{{H}^s})
\end{align*}
The remaining terms is estimated: by the same arguments as above,
\begin{align*}
\|(\nabla A) u \|_{W^{s-1,3/2}} &\lesssim \|\nabla A\|_{H^{s-1}} \|u\|_{L^6} + \|\nabla A\|_{L^2} \|u\|_{W^{s-1,6}} \lesssim F(\|u\|_{H^s})
\end{align*}
Turning to $Vu$, Lemma \ref{lemma V} gives
\begin{align*}
\|S(t-\tau) Vu\|_{H^s} \lesssim \|Vu \|_{H^s} \lesssim \|u\|_{H^s}^3
\end{align*}
        
Finally, we have
\begin{align*}
\left\|\langle D \rangle^{s} S(t-\tau) \frac{(u,Hu)}{\|u_0\|_{L^2}^2} u(\tau) \right\|_{L^2} &\lesssim \|\langle D\rangle^s \frac{(u,Hu)}{\|u_0\|_{L^2}^2} u\|_{L^2} \leq \frac{|(u,Hu)|}{\|u_0\|_{L^2}^2} \|u\|_{H^s} \lesssim \frac{F(\|u\|_{H^1})}{\|u_0\|_{L^2}^2}\|u\|_{H^s},
\end{align*}
where we used Lemma \ref{thm:estimates energy}. Collecting the previous estimates yields
\begin{align*}
    \|\Phi(u(t))\|_{H^s} &\lesssim 
    \|u_0\|_{H^s}+\int_0^t \frac{1}{\epsilon^{1/4}(t-\tau)^{1/4}}\left(1+\frac{1}{\epsilon^{1/2}(t-\tau)^{1/2}}\right) F(\|u\|_{H^s})
    \dd \tau\\
    &\quad + \int_0^t \| u \|_{H^s}^3 \dd \tau + \epsilon\int_0^t \frac{F(\|u\|_{H^1})}{\|u_0\|_{L^2}^2}\|u(\tau)\|_{H^s} \dd \tau.
\end{align*}
Using $\|u\|_{L^{\infty}_T H^{s}} < R$ and the fact that $\|u_0\|_{L^2} > 0$ yields
\begin{align*}
\|\Phi(u(t))\|_{X^s_T} \lesssim \|u_0\|_{H^s}  + F(R) \left(\frac{1}{\epsilon^{1/4}}T^{3/4}+\frac{1}{\epsilon^{3/4}}T^{1/4} \right)+ T R^3 + \epsilon F(R) T.
\end{align*}
for some continuous function $F$. This means we can find a $T_{\epsilon}^{\ast}\leq T_{u_0}$ such that $\Phi$ maps $X_{T_{\epsilon}^{\ast}}^s$ onto itself. By a similar reasoning there is a $T_{\epsilon}<T_{\epsilon}^*$ such that $\Phi$ is a contraction on $X_{T_{\epsilon}}^s$. By Banach's fixed point theorem we deduce that there is a unique solution $u\in C([0,T_{\epsilon}),H^s(\mathbb{R}^3))$ to the Pauli-Darwin and Pauli-Poisswell equations, respectively.
\end{proof}

\section{Local wellposedness}
\label{sec:lwp}

In this section we prove the unique existence of local solutions to the Pauli-Darwin equation \eqref{eq:PD_Pauli_scaled}-\eqref{eq:PD_data_scaled} and the Pauli-Poisswell equation \eqref{eq:PPW_Pauli_scaled}-\eqref{eq:PPW_data_scaled}. We use energy estimates and prove an a priori bound for $\|u\|_{H^s}$, $s>3/2$, cf. Lemma \ref{thm:grönwall}. We then prove Theorem \ref{thm:local wellposedness} in Section \ref{proof of lwp}. 

\subsection{A priori estimate}

\begin{lemma}
\label{thm:grönwall}
Let $u$ be a solution in ${C}^1_t H^s_x$ of either the Pauli-Darwin equation or the Pauli-Poisswell equation, and let $s > \frac 32$. Then we have the following differential inequality

\begin{equation}
\label{eq:grönwall2}
\frac d {dt}\|u(t)\|_{H^s}^2 \lesssim F(\| u \|_{H^s})
\end{equation}
for a continuous function $F$.

{This estimate remains valid for the regularized equations \eqref{eq:new regularization Darwin}-\eqref{eq:new regularization Darwin data} and \eqref{eq:new regularization Poisswell}-\eqref{eq:new regularization Poisswell data}, uniformly in $\epsilon$.}
\end{lemma}

\begin{proof} {We only treat the case of solutions of the equations without regularization, the case of the regularized equations being a straightforward adaptation.}
Apply $D^s$ to \eqref{eq:PD_Pauli}, resp. \eqref{eq:PPw_Pauli}, then take the $\mathbb{C}^2$ inner product with $D^s {u}$ and integrate over $\mathbb{R}^3$:
\begin{align*}
 &   \int \langle D^s {u} , \partial_t D^s u\rangle \dd x - i \int \langle D^s {u} , \Delta D^s u\rangle \dd x \\
    & \; = \int \langle D^s {u} , D^s (A\cdot \nabla u)\rangle \dd x + \frac{1}{2}\int \langle D^s u , D^s ((\text{div} A) u )\rangle \dd x - \frac{i}{2} \int \langle D^s {u} , D^{s} (|A|^2 u)\rangle \dd x \\
    & \quad \quad \qquad \qquad  - i \int \langle D^s {u} , D^s (Vu)\rangle \dd x + \frac{i}{2} \int D^s \langle{u} , D^s ((\sigma \cdot B)u)\rangle \dd x
\end{align*}
Taking the real part and using $2\Re \langle{f}, \partial_t f\rangle = \partial_t |f|^2$ yields
\begin{equation*}
\begin{split}
    \frac{1}{2}\frac{d}{dt} \int |D^s u|^2 \dd x &= \Re \int \langle D^s {u} , D^s (
    A\cdot \nabla u)\rangle \dd x + \frac{1}{2}\Re\int \langle D^s u , D^s ((\text{div} A) u )\rangle \dd x \\ & \qquad - \Re\frac{i}{2} \int\langle D^s {u} , \nabla^{s} (|A|^2 u)\rangle \dd x - \Re i \int \langle D^s {u} , D^s (Vu)\rangle \dd x \\ & \qquad + \Re\frac{i}{2} \int \langle D^s {u} , D^s ((\sigma \cdot B)u)\rangle \dd x \\
& = I + II + III + IV + V.
\end{split}
\end{equation*}

The terms $II$, $III$, $IV$, $V$ can be estimated as follows: on the one hand, by Lemma \ref{lemma V},
\begin{equation*}
\left\| Vu  \right\|_{H^s} \lesssim \|u\|_{H^1}^2 \|u\|^2_{H^s}.
\end{equation*}
On the other hand, by the fractional Leibniz rule, Lemma \ref{lem:apriori_magnetic} and the Sobolev embedding theorem,
\begin{align*}
  &  \|D^s(|A|^2u)\|_{L^2} \lesssim \|A\|_{L^\infty}(\|A\|_{\dot{H}^s}\|u\|_{L^\infty}+ \|A\|_{L^\infty}\|u\|_{\dot{H}^s}) \lesssim F(\|u\|_{H^s})\\
&  \|D^s((\nabla \cdot A)u)\|_{L^2} + \|D^s((\sigma \cdot B)u)\|_{L^2} \lesssim \|\nabla A\|_{\dot{H}^s} \|u\|_{L^\infty} + \|\nabla A\|_{L^\infty} \|u\|_{\dot{H}^s}
 \lesssim F(\|u\|_{H^s}).
\end{align*}
Using in addition the Cauchy-Schwartz inequality gives
$$
| II|  + | III| + |IV| + | V | \lesssim F(\|u\|_{H^s}).
$$

This leaves us with the term $I$, which we write as
$$
I     = \Re \int \langle D^s {u} , (
    A\cdot \nabla D^s u)\rangle \dd x +  \Re \int \langle D^s {u} , D^s (
    A\cdot \nabla u) - A\cdot \nabla D^s u\rangle \dd x = I' + I''.
$$
To deal with $I'$, we integrate by parts to obtain, with the help of Lemma \ref{lem:apriori_magnetic},
$$
| I' | = - \Re \int \operatorname{div} A  \,| D^s u|^2 \dd x \lesssim \| \nabla A \|_{L^\infty} \| u \|_{H^s}^2 \lesssim F(\| u \|_{H^s}).
$$
To estimate $I''$, we rely on the following commutator estimate from \cite{FMRR}
$$
\left\| D^s \left[ f \cdot \nabla g \right] - f \cdot \nabla D^s g \right\|_{L^2} \lesssim \| \nabla f \|_{H^s} \| g \|_{H^s} \qquad \mbox{if $s > \frac 32$}.
$$
With the help of Lemma \ref{lem:apriori_magnetic}, it implies that
$$
| I''| \lesssim \| u \|_{H^s}^2 \| \nabla A \|_{H^s} \lesssim F(\| u \|_{H^s}).
$$
\end{proof}

\subsection{Proof of Theorem \ref{thm:local wellposedness}}
\label{proof of lwp}

{Proposition \ref{proposition lwp regularized} provides a local strong solution of the regularized equation on an interval depending on $\epsilon$. The a priori estimate \eqref{eq:grönwall2} allows the extension of the solution from the $\epsilon$-dependent interval $(0,T_{\epsilon})$ to an interval denoted by $(0,T)$, whose length depends only on the initial data $u_0$. 

\medskip

We will now show that the limit $\epsilon \rightarrow 0$ furnishes a strong solution. To this end, consider two regularized solutions $u^{\epsilon}$, $u^{\epsilon'}$  of \eqref{eq:new regularization Poisswell}-\eqref{eq:new regularization Poisswell data} with identical initial data $u_0$. We will show that the sequence is Cauchy in $C([0,T],L^2(\mathbb{R}^3))$. Then it follows from interpolation with $L^{\infty}([0,T],H^s(\mathbb{R}^3))$ that $u^{\epsilon}$ is Cauchy in $C([0,T],H^r(\mathbb{R}^3))$ for $r<s$. Using the equation, one has $\partial_t u^{\epsilon} \in C([0,T],H^{r-2}(\mathbb{R}^3))$ and by taking pointwise limits one has that $u\in C([0,T],H^r(\mathbb{R}^3)) \cap C^1([0,T],H^{r-2}(\mathbb{R}^3))$ solves the equation. From Fatou's lemma, the boundedness of $u^{\epsilon}$ in $H^s$ and the convergence to $u$ in $L^2$ it follows that $u\in L^{\infty}([0,T],H^s(\mathbb{R}^3))$. 

To see that $u^{\epsilon}$ is Cauchy in $C([0,T],L^2(\mathbb{R}^3))$, consider the difference $v= u^{\epsilon}-u^{\epsilon'}$ which solves
\begin{align*}
    \partial_t v &= \frac{1}{2}(\epsilon - \epsilon')\Delta u^{\epsilon} + \frac{\epsilon'}{2}\Delta v + \frac{i}{2}\Delta v + K_{\epsilon}(u^{\epsilon})-K_{\epsilon'}(u^{\epsilon'}) \\ &\qquad \qquad + \epsilon \frac{(u^{\epsilon},H^{\epsilon}u^{\epsilon})}{\|u_0\|_{L^2}^2}u^{\epsilon}-\epsilon' \frac{(u^{\epsilon'},H^{\epsilon'}u^{\epsilon'})}{\|u_0\|_{L^2}^2}u^{\epsilon'}.
\end{align*}
Taking the $\mathbb{C}^2$ inner product with $\overline{v}$ and integrating yields}
\begin{align*}
    \partial_t ||v||_2^2 &\leq |(\epsilon - \epsilon')|\left|\int (\Delta u^{\epsilon} )\overline{v}\dd x\right| - \epsilon'\|\nabla v\|_2^2 \\ &\qquad + \int \langle{v},K_{\epsilon}(u^{\epsilon})-K_{\epsilon'}(u^{\epsilon'})\rangle  + \int \langle{v},\epsilon \frac{(u^{\epsilon},H^{\epsilon}u^{\epsilon})}{\|u_0\|_{L^2}^2}u^{\epsilon}-\epsilon' \frac{(u^{\epsilon'},H^{\epsilon'}u^{\epsilon'})}{\|u_0\|_{L^2}^2}u^{\epsilon'}\rangle
\end{align*}
For the first term, note that since $H^{3/2+} \subset H^{-3/2-}$ and by the a priori bound for $u^{\epsilon}$ in $H^{3/2+}$,
\begin{equation*}
    |(\epsilon - \epsilon')|\left|\int (\Delta u^{\epsilon} )\overline{v}\dd x\right| \leq |(\epsilon - \epsilon')\|u^{\epsilon}\|_{H^{3/2+}}\|v\|_{H^{-1/2-}}
\end{equation*}
The second term is negative and the remaining terms can be easily estimated using Lemmas \ref{lemma V}, \ref{lem:apriori_magnetic} and \ref{thm:estimates energy}. 

By Grönwall we obtain
\begin{equation}
    \sup_{t\in [0,T]} \|v\|_{L^2} \leq |\epsilon - \epsilon'|
\end{equation}
This implies that $u^{\epsilon}$ is a Cauchy sequence in $C([0,T],L^2)$.

Uniqueness of solutions follows from a similar argument using the difference of two solutions $u-w$ with identical initial data  and $\epsilon = \epsilon'$.

The continuous dependence on the data follows by a standard argument due to J. Bona and R. Smith, cf. \cite{bona1975initial, erdougan2016dispersive}, which we sketch for the sake of completeness. 

Define the approximate identity $\varphi^{\delta}$, $\delta > 0$ where $\varphi \in \mathcal{S}$ such that $\varphi(0) = 1$, $\partial_{\xi}^k \widehat{\varphi}(0) = 0$ for all $k>0$ and $\varphi^{\delta}(x) = \delta^{-1} \varphi(x/\delta)$. Then define the mollified initial data $u_0^{\delta} = \varphi^{\delta} \ast u_0$. Choose a sequence $u_{n}$ of solutions with initial data $u_{n,0}$ such that $u_{n,0}$ converges to $u_0$ in $H^s$. The goal is to show that $u_n$ converges to a solution $u$ with initial data $u_0$ uniformly in time for $t\leq T$. We have
\begin{equation}
    \|u-u_n\|_{H^s} \leq \|u-u^{\delta} \|_{H^s} + \|u^{\delta} - u^{\delta}_n\|_{H^s} + \|u^{\delta}_n - u_n \|_{H^s}.
\end{equation}
Here, $u^{\delta}_n$ is the solution with data $u^{\delta}_{n,0}=\varphi^{\delta} \ast u_{n,0}$. 

The first term can be treated as follows:
Observe that $\|u^{\delta}\|_{H^{s+l}} \leq \delta^{-l}$, which follows from the a priori estimate and the definition of $u^{\delta}_0$. Therefore $u^{\delta}$ is a $C([0,T],H^s)$-solution by interpolation. Furthermore, for $\delta'<\delta$, we have\begin{equation}
\|u^{\delta}-u^{\delta'}\|_{L^2} = o(\delta^{s}).
\label{eq:difference delta delta'}
\end{equation} This follows from
\begin{equation*}
    \|u^{\delta}-u^{\delta'}\|_{L^2} \lesssim  \|u^{\delta}_0 - u^{\delta'}_0\|_{L^2},
    \label{eq:a priori differences}
\end{equation*}
and the definitions of $u_0^{\delta},u_0^{\delta'}$. By interpolation with the bound on $u^{\delta}$ in $H^{s+l}$ we obtain that $u^{\delta}$ is a convergent sequence in $L^{\infty}([0,T],H^s)$ converging to $u\in C([0,T],H^s)$.

For the second term we have
\begin{equation}
    \|u^{\delta}-u^{\delta}_n\|_{H^s} \lesssim e^{{T}{\delta^{-l}}} \|u^{\delta}_0-u^{\delta}_{n,0}\|_{H^s}
\end{equation}
which can be proved by calculating $\partial_t\|D^s(u^{\delta}-u^{\delta}_n)\|_{L^2}^2$ and using $\|u^{\delta}\|_{H^{s+l}}\leq \delta^{-l}$ together with Grönwall's inequality.

For the third term, we need to show 
\begin{equation}
    \sup_n \|u^{\delta}_n-u_n\|_{H^s} \rightarrow 0
\end{equation}
as $\delta \rightarrow 0$. To this end consider $\delta'<\delta$. Since $\|u^{\delta}_n-u_n\|_{H^s} = \lim_{\delta'\rightarrow 0}\|u^{\delta}_n-u^{\delta'}_n\|_{H^s}$ we only need to show that
\begin{equation}
\sup_{\delta'<\delta,n} \|u^{\delta}_n-u_n^{\delta'}\|_{L^2} = o(\delta^s),
\label{eq:estimate third term}
\end{equation}
Then interpolation with $\|u^{\delta}\|_{H^{s+l}}\leq \delta^{-l}$, yields the claim. In order to show \eqref{eq:estimate third term}, we can again resort to showing it for the initial data (by the a priori estimate) similarly to the first term.

\section{Global finite energy weak solutions}
\label{sec:weak sol}

In this section we will prove the global existence of weak solutions in the energy space $H^1(\mathbb{R}^3)$.
First we extend the local solutions of Proposition \ref{thm:local wellposedness} globally. Then we take the limit $\epsilon \rightarrow 0$, which will be justified by uniform bounds for the regularized solution.

\subsection{Global solutions of the regularized equations}
\label{sec:regularized_wellposedness}

First we extend the local solution of Proposition \ref{proposition lwp regularized} to a global one:
\begin{proposition}
    \label{proposition gwp regularized H^1}
Let $\epsilon>0$, $s = 1$. Then the solutions $u^\star_\epsilon \in C((0,T_{\epsilon}),H^1(\mathbb{R}^3))$, $\star =\text{D},\text{PW}$ constructed in Proposition \ref{proposition lwp regularized} can be extended globally to $C(\mathbb{R},H^1(\mathbb{R}^3))$. Moreover, the solutions are bounded in $H^1(\mathbb{R}^3)$ uniformly in $\epsilon$ and $t$.
\end{proposition}



\begin{proof}

Owing to the blow-up alternative we deduce that $T_{\epsilon}=\infty$ if the $H^1$-norm of $u$ does not blow up in finite time. By Lemma \ref{thm:estimates energy} the energy is bounded a priori at time $t=0$, i.e. for some function $F$,

\begin{equation}
    \label{eq:energy_at_time_zero}
    E(0) < F(\|u_0\|_{H^1}).
\end{equation}

Note that since $(\sigma \cdot \nabla)^2 = \Delta$, we have $\| \nabla u \|_{L^2} = \| \sigma \cdot \nabla u \|_{L^2}$. By Lemma \ref{thm:decay lemma}, the charge and the energy are decaying; using furthermore H\"older's inequality and the Sobolev embedding theorem,
\begin{align*}
\|\nabla u\|_{L^2} &\leq \|(\sigma \cdot \nabla_A )u\|_{L^2} + \|(\sigma \cdot A)u\|_{L^2}  \\
        &\leq E(0)^{1/2} +\|A\|_{L^6} \|u\|_3 \\
        &\leq E(0)^{1/2} + C \|\nabla A\|_{L^2} \|u\|_{L^2}^{1/2} \|u\|_{L^6}^{1/2} \\
        &\leq E(0)^{1/2} + C E(0)^{1/2}Q(0)^{1/4} \|\nabla u \|_{L^2}^{1/2},
\end{align*}
{By the Cauchy-Schwarz inequality, $C E(0)^{1/2}Q(0)^{1/4} \|\nabla u \|_{L^2}^{1/2} \leq \frac{1}{2} \|\nabla u \|_{L^2} +  \frac{C^2}{2} E(0)Q(0)^{1/2}$, which entails
\begin{equation}
    \|\nabla u \|_{L^2} \lesssim E(0)^{1/2} + E(0)Q(0)^{1/2}.
\end{equation}}
 Together with  and \eqref{eq:energy_at_time_zero} one obtains the a priori bound $\|u\|_{H^1} \leq C$. Now iterate the solution by a standard continuation argument.
 \end{proof}

\subsection{Global weak solutions of the original equation}
\label{sec:global_weak_solutions}
\begin{proof}[Proof of Theorem \ref{th:global weak solutions}]
To simplify notations, we will focus on the case of the Pauli-Poisswell equation, since the Pauli-Darwin equation can be treated almost identically.

\medskip

\noindent \underline{Meaning of the PDE.}
Let us first specify the sense in which this PDE will be solved. It can be written
\begin{align}
i\partial_t u &= -\frac{1}{2} \Delta u + iA\cdot \nabla u + \frac{i}{2} (\text{div} A) u + \frac 12 |A|^2 u + V u -\frac{1}{2}  (\mathbf{\sigma} \cdot B) u, \\
-\Delta V &= |u|^2,  \\
-\Delta A&= \Im \langle u, \nabla u\rangle - A |u|^2 + \frac{1}{2}\nabla \times \langle u, \mathbf{\sigma} u \rangle , \qquad B = \nabla \times A, \\
u(0,x) &= u_0,
\end{align}
We will require $(u,A,V)$ to be of finite energy, namely
$$
(u,A,V)  \in L^\infty_t H^1_x \times L^\infty_t \dot H^1_x \times L^\infty_t \dot{H}^1_x, \qquad (t,x) \in \mathbb{R}_+ \times \mathbb{R}^3.
$$
This allows to give a meaning to all the nonlinear terms appearing above as products between Lebesgue functions. Then the equalities above are required to hold in the sense of distributions.

\bigskip

\noindent \underline{Uniform bounds on approximate solutions} Consider a sequence $\epsilon_n$ such that $\epsilon_n \rightarrow 0$ as $n\rightarrow \infty$. By Proposition \ref{proposition gwp regularized H^1}, there exists a unique solution $u_n$ for each $n$ to the regularized Pauli-Poisswell equation

\begin{align*}
i\partial_t u_n &= (1 -i \epsilon_n ) \left[ -\frac{1}{2} \Delta u_n + iA\cdot \nabla u_n + \frac{i}{2} (\text{div} A_n) u_n + \frac 12 |A_n|^2 u_n + V_n u_n  -\frac{1}{2}  (\mathbf{\sigma} \cdot B_n)\right]u_n,\\
&\qquad -i\epsilon_n \frac{(u_n,H_n,u_n)}{\|u_0\|_{L^2}^2}u_n \\
-\Delta V_n &= |u_n|^2,  \\
-\Delta A_n&= \Im \langle u_n, \nabla u_n \rangle - A_n |u_n|^2 + \frac{1}{2}\nabla \times \langle u_n, \mathbf{\sigma} u_n \rangle , \qquad B_n = \nabla \times A_n, \\
u_n(0,x) &= u_0  ,
\end{align*}
where 
\begin{equation}
    H_n = -\frac{1}{2}(\sigma \cdot (\nabla-iA_n))^2 + V_n u_n
\end{equation}
is the Pauli Hamiltonian.

Let us now record the uniform bounds which hold for the sequence $(u_n,A_n,V_n)$.
\begin{itemize}
\item By Proposition \ref{proposition gwp regularized H^1}, the sequence $(u_n,A_n,V_n)$ is uniformly bounded in $(L^\infty_t H^1_x \times L^\infty_t \dot H^1_x \times L^\infty_t \dot{H}^1_x)(\mathbb{R}_+ \times \mathbb{R}^3)$.
\item By examining the right-hand side of the evolution equation for $(u_n)$, and using classical arguments (H\"older's inequality and Sobolev embedding theorem), it appears that $\partial_t u_n$ is uniformly bounded in $L^\infty_t H^{-1}(\mathbb{R}_+ \times B_{\mathbb{R}^3}(0,R))$ for any $R>0$.
\end{itemize}

\bigskip

\noindent \underline{Compactness and extraction of the limit} By the Banach-Alaoglu theorem, there exists $(u,A,V) \in L^\infty_t H^1_x \times L^\infty_t \dot H^1_x \times L^\infty_t \dot{H}^1_x$ such that
$$
(u_n,A_n,V_n) \overset{w-*}{\longrightarrow} (u,A,V) \qquad \mbox{in $L^\infty_t H^1_x \times L^\infty_t \dot H^1_x \times L^\infty_t \dot{H}^1_x$ as $n\to \infty$},
$$
up to selecting a subsequence (which we still denote by the subscript $n$ for simplicity).

Furthermore, thanks to the uniform bounds on $(u_n)$ and $(\partial_t u_n)$ recorded above, we can apply the Aubin-Lions lemma to this convergent subsequence. It gives a new subsequence such that
$$
u_n \longrightarrow u \qquad \mbox{in $L^p_{t,x} ([0,R] \times B_{\mathbb{R}^3}(0,R))$, $p<6$, as $n\to \infty$, for any $R>0$}
$$
(to be more precise: the Aubin-Lions lemma gives convergence for a fixed $R>0$ and the convergence for any $R>0$ follows from a diagonal argument).

Combining the above with Lemma \ref{th:strong_convergence_A} gives that
$$
A_n \longrightarrow A \qquad \mbox{in $L^6_{t,x} ([0,R] \times B_{\mathbb{R}^3}(0,R))$, as $n\to \infty$, for any $R>0$}
$$
The same holds for $V$ with a much simpler proof (see for instance \cite{guo1995global}), which we omit:
$$
V_n \longrightarrow V \qquad \mbox{in $L^6_{t,x} ([0,R] \times B_{\mathbb{R}^3}(0,R))$, as $n\to \infty$, for any $R>0$}
$$

\medskip

\noindent \underline{Passing to the limit in the equation} The uniform bounds which were given above immediately imply that
\begin{equation*}
\epsilon_n \left[ -\frac{1}{2} \Delta u_n + iA\cdot \nabla u_n + \frac{i}{2} (\text{div} A_n) u_n + \frac 12 |A_n|^2 u_n + V_n u_n +\frac{(u_n,H_n,u_n)}{\|u_0\|_{L^2}^2}u_n\right] \overset{n \to \infty}{\longrightarrow} 0 \qquad \mbox{in $\mathcal{S}'$}.
\end{equation*}
There remains to take the limit of the nonlinear terms. Those which do not involve derivatives are easily dealt with by the strong convergence properties which were mentioned. Those which do involve derivatives can be written as the product of a strongly converging term and a weakly converging term, which allows to pass to the limit.
\end{proof}

\section*{Acknowledgement}
N.M. and J.M. acknowledge financial support from the Austrian Science Fund (FWF) via the SFB project F65 and the DK project W1245, as well as  the Vienna Science and Technology Fund (WWTF) via project MA16-066 "SEQUEX".

J.M. acknowledges financial support by the Austrian Science Fund (FWF) 10.55776/J4840.

P.G. was supported by the Royal Society through a Wolfson fellowship and is grateful to the Wolfgang Pauli Institute for continuous hospitality.

\bibliography{wellposedness}
\bibliographystyle{abbrv}

\end{document}